\documentclass[11pt]{amsart}

\usepackage[english]{babel}
\usepackage{enumitem}
\usepackage[ansinew]{inputenc}
\usepackage{amsmath,amsthm}
\usepackage{hyperref}
\usepackage[T1]{fontenc}
\usepackage{lmodern}
\usepackage{cleveref}
\usepackage{amssymb}
\usepackage{mathtools}
\usepackage{amscd}
\usepackage{marginnote}
\usepackage[all]{xy} % XY-Matrix
\usepackage{dsfont}%einheitsmatrix
\usepackage{tikz-cd}
\usepackage{stmaryrd} % Eckige Klammern [[ \llbracket \rrbracket ]]

\setenumerate[1]{label=(\alph*)}
\setenumerate[2]{label=(\roman*)}

\theoremstyle{plain}
\newtheorem{thm}{Theorem}[section]
\newtheorem*{thm*}{Theorem}
\newtheorem{lem}[thm]{Lemma}
\newtheorem{cor}[thm]{Corollary}
\newtheorem*{cor*}{Corollary}
\newtheorem{prop}[thm]{Proposition}

\theoremstyle{remark}

\theoremstyle{definition}

\newtheorem{defin}[thm]{Definition}

\newtheorem{notation}[thm]{Notation}

\newcommand{\ZZ}{\mathbb{Z}}

\newcommand{\Zp}{\mathbb{Z}_p}
\newcommand{\Qp}{\mathbb{Q}_p}
\newcommand{\Cp}{\mathbb{C}_p}
\newcommand{\CC}{\mathbb{C}}
\newcommand{\QQ}{\mathbb{Q}}
\newcommand{\lcm}{\mathrm{lcm}}

\newcommand{\righteq}{\stackrel{\sim}{\rightarrow}}
\newcommand{\Ocal}{\mathcal{O}}

\newcommand{\dn}[1]{d_{#1}}

\AtBeginDocument{%
   \def\MR#1{}
}

\title[Linear independence result for p-adic L-values]{A Linear independence result for $p$-adic $L$-values}

\author{Johannes Sprang}
\thanks{The author has been supported by the SFB 1085 ``Higher Invariants'' funded by the DFG}
\address{Fakult\"at f\"ur Mathematik Universit\"at Regensburg  \\ 93040 Regensburg }
\email{johannes.sprang@mathematik.uni-regensburg.de }
\date{}

\hypersetup{
pdfauthor = {J. Sprang},
pdftitle = {Linear independence of p-adic zeta values}
}

\begin{document}
\begin{abstract}
The aim of this paper is to provide an analogue of the Ball--Rivoal theorem for $p$-adic $L$-values of Dirichlet characters. More precisely, we prove for a Dirichlet character $\chi$ and a number field $K$ the formula $\dim_{K}\left(K+\sum_{i=2}^{s+1} L_p(i,\chi\omega^{1-i}) K \right)\geq \frac{(1-\epsilon)\log (s)}{2[K:\QQ](1+\log 2)}$. As a byproduct, we establish an asymptotic linear independence result for the values of the $p$-adic Hurwitz zeta function.
\end{abstract}

\maketitle

\section{Introduction}
The values of the Riemann zeta function at even positive integers are non-zero rational multiples of powers of $\pi$ and thereby transcendental. The question about the structure of the odd positive zeta values is much more difficult and has not yet been sufficiently clarified. The first result in this direction was the proof of the irrationality of $\zeta(3)$ given by Ap\'ery in \cite{apery}. While this is still the only particular odd zeta value known to be irrational, the celebrated theorem of Ball and Rivoal gives an asymptotic lower bound for the dimension of the $\QQ$-vector space spanned by the first odd zeta values between $3$ and $s$:
\begin{thm*}[{Rivoal, Ball--Rivoal, \cite{rivoal,ball_rivoal}}]
For any $\epsilon>0$, there exists an integer $s_0$ such that for all odd integers $s\geq s_0$:
\[
	\dim_{\QQ}\left(\QQ+ \zeta(3)\QQ+\zeta(5)\QQ+\dots+\zeta(s)\QQ \right)\geq \frac{\log(s)(1-\epsilon)}{1+\log(2)}.
\]
\end{thm*}
The lower bound on the dimension has not been significantly improved since then. Nevertheless, if one is only interested in the \emph{irrationality} of odd zeta values and not in 		their \emph{linear independence}, the lower bound can be further improved.
\begin{thm*}[{Fischler--Sprang--Zudilin, \cite{FSZ}}]
For any $\epsilon>0$, there exists an integer $s_0$ such that for all odd integers $s\geq s_0$ at least
\[
	2^{(1-\epsilon)\frac{\log(s)}{\log\log(s)}}
\]
among the numbers
\[
	\zeta(3),\zeta(5),\dots,\zeta(s)
\]
are irrational.
\end{thm*}
More recently, S. Fischler gives a unified approach to both theorems using a variant of Shidlovsky's Lemma \cite{fischler_L}. More importantly, he has  extended both results to the $L$-values of Dirichlet characters and general $N$-periodic functions and has thereby generalized the previous work of Nash and Nishimoto \cite{nash,nishimoto}.\par
Let us now turn our attention to results on $p$-adic $L$-values for a prime $p$. Although we are still far away from having a good understanding of the nature of odd zeta values, even less is known about the values of the $p$-adic zeta function at positive integers. Building on a beautiful re-interpretation of Ap\'ery's proof by Beukers in terms of modular forms, Calegari succeeded in proving the irrationality of $\zeta_p(3)$ for $p=2,3$ and of $L_2(2,\chi)$ with $\chi$ the Dirichlet character of conductor $4$, \cite{calegari}. Shortly afterwards, Beukers gave a new proof of Calegari's results and furthermore provided irrationality proofs for certain values of the $p$-adic Hurwitz zeta function at small integers and for small primes $p$, cf.~\cite{beukers}. Partial results towards a $p$-adic analogue of the Ball--Rivoal theorem have been given by Bel, \cite{bel}:
\begin{thm*}[Bel]
	Let $d$ and $s$ be positive integers and $x\in \frac{1}{d}\ZZ$, then
	\[
		\dim_{\QQ}\left( \QQ+ \zeta_p(2,x)\QQ+\dots+\zeta_p(s,x)\QQ\right)\geq \frac{s \log|x|_p}{\log d+\sum_{l|d}\frac{\log l}{l-1} +s+(s-1)\log 2},
	\]
	if $|x|_p$ is sufficiently large.
\end{thm*}
More generally, Bel proves similar results for certain $p$-adic Lerch zeta values \cite[Theorem 3.4.]{bel}.
Note that this theorem only applies if the modulus $x$ of the $p$-adic Hurwitz zeta function has a sufficiently large $p$-adic norm. Similar results have been obtained by Hirose, Kawashima and Sato in \cite{HKS}.\par
In the following, we write $\Cp$ for the completion of a fixed algebraic closure of the field of $p$-adic numbers $\Qp$. Let $\bar{\QQ}$ be the algebraic closure of $\QQ$ in $\Cp$ and let $K\subseteq \bar{\QQ}$ be a number field. For a Dirichlet character $\chi$, let us write $\QQ(\chi)\subseteq \bar{\QQ}$ for the smallest subfield of $\bar{\QQ}$ containing all values of $\chi$. Let us recall that we have the following $p$-adic interpolation property for the $p$-adic $L$-function of a Dirichlet character $\chi$:
\[
	L_p(i,\chi\omega^{1-i})=(1-\chi(p)p^{-i})L(i,\chi),\quad i\in\ZZ_{<0},
\]
where  $\omega$ is the Teichm\"uller character. This formula suggests the question of	 the structure of the $p$-adic integers $L_p(i,\chi\omega^{1-i})$ for \emph{positive} integers $i\in\ZZ$. The following result is the main theorem of this paper:
\begin{thm}\label{thm_main}
For $\epsilon>0$ and $s$ a sufficiently large positive integer, we have
\[
	\dim_{K}\left( K+\sum_{i=2}^{s}L_p(i,\chi\omega^{1-i})K  \right)\geq \frac{(1-\epsilon)\log s}{2[K:\QQ]\left( 1+\log(2)\right)}.
\]
\end{thm}
In the classical situation, half of the positive $L$-values are 'trivially' transcendental and one usually excludes them by a parity condition. In the $p$-adic world, the situation is different. Half of the positive $p$-adic $L$-values vanish anyhow: Indeed, this follows from the vanishing of $L(1-s,\chi)$ at half of the negative integers and the interpolation property. So, contrary to the classical case, we do not have to impose a parity condition.\par
The special case $\chi=1$ gives a linear independence criterion for $p$-adic zeta values:
\begin{cor}
	For every $\epsilon>0$ and $s$ sufficiently large, we have
	\[
	\dim_{K}\left( K+\sum_{i=2}^{s}\zeta_p(i)K  \right)\geq \frac{(1-\epsilon)\log s}{2[K:\QQ](1+\log(2))}.
	\]
\end{cor}
Here, we use the definition
\[
	\zeta_p(s):=L_p(s,\omega^{1-s})=\lim_{k\rightarrow s} \zeta(k),\quad \text{for } s\in\ZZ_{>1},
\]
where $k$ runs over a sequence of negative integers  which are congruent to $s$ modulo $p-1$ and converge $p$-adically to $s$, see \cite[Lemma 2.4]{calegari}. This is motivated by the fact that the functions
\[
	s\mapsto L_p(s,\omega^{1-k_0}),
\]
for $k_0\in \{1,\dots,p-1\}$, are exactly the $(p-1)$ branches interpolating the values of the Riemann zeta function at negative integers.

Furthermore, we obtain the following result for the $p$-adic Hurwitz zeta values as a byproduct:
\begin{thm}\label{thm_Hurwitz}
Let $x\in \QQ$ with $|x|_p>1$. For $\epsilon>0$ and $s$ sufficiently large, we have
\[
	\dim_{K}\left( K+\zeta_p(2,x)K+\dots+\zeta_p(s,x)K  \right)\geq \frac{(1-\epsilon)\log s}{2[K:\QQ](1+\log(2))}.
\]
\end{thm}
 Our strategy is roughly, to construct linear forms of $p$-adic $L$-values, using certain linear combinations of linear forms of $p$-adic Hurwitz $L$-values with related coefficients. Many ideas of our paper are motivated by Bel's paper, so let us briefly recall his strategy. Bel starts with the construction of linear forms in values for the classical Hurwitz zeta function via Pad\'e approximation. Then, he observes that a variant of the Euler--Maclaurin formula
\[
	\zeta(s,x)=\frac{x^{1-s}}{s-1}-\sum_{j=1}^k\binom{-s}{j-1}\frac{B_j}{j}x^{1-s-j}+O(x^{-s-j}) \text{ as } x\rightarrow \infty,
\]
is also valid $p$-adically. This allows him to transfer the linear forms in values of the classical Hurwitz zeta function to the $p$-adic world. In this way, he obtains a sufficient number of linearly independent families of linear forms, to which he finally applies a $p$-adic variant of Siegel's criterion. In this paper, we avoid the passage from the Archimedean to the non-Archimedean world and work intrinsically using $p$-adic methods. To outline our approach, let us briefly recall that the most important source of linear forms in Hurwitz zeta values is provided by sums of the form
\[
	\sum_{k\geq 0} R_n(k+x)=L_n(1,\zeta(3,x),\dots,\zeta(s,x)),\quad x\in\QQ,
\]
with $R_n(t)\in\QQ(t)$ a rational function of degree less then $1$ having only poles of maximal order $s$ at negative integers. Indeed, this construction appears in all of the above-mentioned results on zeta values. The formula 
\[
		\omega(x)^{1-s}\zeta_p(s,x)=\frac{1}{s-1}\int_{\Zp} \frac{1}{(x+t)^{s-1}}\mathrm{d}t:=\lim_{n\rightarrow\infty}\frac{1}{p^n}\sum_{1\leq k< p^n}\frac{1}{(x+k)^{s-1}},
\]
(see \Cref{lem_zeta_Volkenborn}) suggests that \emph{Volkenborn integration} over rational functions is a right $p$-adic analogue to the above summation process.\par
\vfill
Our paper is organized as follows: After recalling basic facts about $p$-adic Hurwitz zeta functions and Volkenborn integration, we use Volkenborn integration over certain rational functions $R_n(t)\in\QQ(t)$ to construct families of linear forms in $p$-adic Hurwitz zeta values (see \cref{sec_1})
\begin{equation}\label{eq_intro1}
		\int_{\Zp} R_n(t+x) \mathrm{d}t=\rho^{(n)}_{0,x}+\sum_{i=1}^s \rho_{i}^{(n)}\omega(x)^{-i}\zeta_p\left(i+1,x\right).
\end{equation}
Here, $x\in \QQ$ is a rational number of sufficiently large $p$-adic norm. This yields families of linear forms with \emph{related coefficients} in the sense that the higher coefficients $\rho_i^{(n)}$, $i\geq 1$ are independent of the parameter $x\in\Qp$. In \cref{sec_1_1}, we study the arithmetic properties of the coefficients $\rho_i^{(n)}$. In \cref{sec_1_2}, we discuss the Archimedean growth of the coefficients $\rho_i^{(n)}$ and \cref{sec_1_3} contains the $p$-adic convergence properties of the linear forms. Since the higher coefficients are independent of $x$, certain linear combinations of \eqref{eq_intro1} yield linear forms in $p$-adic $L$-values. Finally, in \cref{sec_2}, we apply  the following variant of Nesterenko's $p$-adic linear independence criterion to conclude the proof.
\begin{thm}\label{thm_nesterenko}
	Let $\tau_1,\tau_2,\tau$ be positive real numbers and $\sigma(n)$ a non-decreasing positive function satisfying $\lim_{n\rightarrow \infty}\sigma(n)=\infty$ and $\lim_{n\rightarrow\infty}\frac{\sigma(n+1)}{\sigma(n)}=1$. Let $\underline{\theta}=(\theta_1,\dots,\theta_s)\in\Cp^s$. Assume that for all sufficiently large integers $n$, there exists a linear form with coefficients in the ring of integers $\Ocal_K$ of the number field $K$ in $s+1$ variables
	\[
		\Lambda_n(\underline{X})=\lambda_{0,n}X_0+\lambda_{1,n}X_1+\dots+\lambda_{s,n}X_m,\quad \lambda_{i,n}\in\Ocal_K
	\]
	satisfying
	\[
		H_K(\Lambda_n)\leq e^{\sigma(n)(\tau+o(1))} \quad\text{ and } e^{-(\tau_1+o(1))\sigma(n)}\leq |\Lambda_n(1,\underline{\theta})|_p\leq e^{-(\tau_2-o(1))\sigma(n)}.
	\]
	Then 
	\[
		\dim_K(K+K\theta_1+\dots+K\theta_s)\geq \frac{\tau_1}{\tau+\tau_1-\tau_2}.
	\] 
	Here, the height of the linear form $\Lambda_n$ is defined as
	\[
	H_K(\Lambda_n):=\max_{0\leq i\leq s} |N_K(\lambda_{i,n})|,
	\]
	with $N_K\colon K\rightarrow \QQ$ the norm map.
\end{thm}
This variant of Nesterenko's criterion is essentially based on Chantanasiri's work, but the proof for general number fields does not appear in the literature. We include the proof in \cref{ch_linearindependence}.\par
We hope that the methods developed in this paper, especially the technique of Volkenborn integration over suitable rational functions, proves to be useful in further investigations of $p$-adic $L$-values.

\section*{Acknowledgement}
It is a pleasure to thank Pierre Bel for pointing out a mistake in an early draft of this paper. I would like to thank Pierre Bel and Tanguy Rivoal for interesting and inspiring discussions in Lyon. Furthermore, I am grateful for many valuable comments and remarks of St\'ephane Fischler and Wadim Zudilin. Last but not least, I would like to thank the anonymous referees for all the valuable suggestions, which have improved the quality of the paper considerably.  

\section{Volkenborn integration and wavelets}
The technique of Volkenborn integration is our most important ingredient for the construction of linear forms in $p$-adic $L$-values. The theory of ($p$-adic) wavelets, having its origin in $p$-adic analysis, turns out to be useful for bounding the $p$-adic growth of the involved linear forms.
In this section, we recall and develop the necessary properties of Volkenborn integration and wavelets.\par

Volkenborn integration has been developed by Volkenborn in \cite{volkenborn_1}. Let us write $C(\Zp,\Qp)$ for the space of continuous functions from $\Zp$ to $\Qp$ and let $f\in C(\Zp,\Qp)$. The function $f$ is considered to be \emph{Volkenborn integrable} if the sequence
\[
	 \frac{1}{p^n}\sum_{0\leq k<p^n}f(k)
\]
converges $p$-adically. In this case, the value
\[
	\int_{\Zp} f(t) \mathrm{d}t:=\lim_{n\rightarrow\infty} \frac{1}{p^n}\sum_{0\leq k<p^n}f(k)
\]
is called the \emph{Volkenborn integral} of $f$. For example, continuously differentiable functions and analytic functions are Volkenborn integrable, see \cite[\S 55]{schikhof}.

The Volkenborn integral is not translation invariant, indeed the following holds.
\begin{prop}[{\cite[Satz 3]{volkenborn_1}}]\label{prop_volkenborn_transl}
	Let $m\in\ZZ_{\geq 0}$, and let $f\colon \Zp\rightarrow\Qp$ be Volkenborn integrable. Then
	\[
		\int_{\Zp} f(u+m) \mathrm{d}u=\int_{\Zp}f(u)\mathrm{d}u +\sum_{i=0}^{m-1}f'(i).
	\]
\end{prop}

Let us now turn our attention to the orthonormal basis of wavelets for $C(\Zp,\Qp)$. The orthonormal basis of wavelets has been introduced by van der Put. Here, we emphasize certain properties of wavelets which are useful for computing Volkenborn integrals. For $k\in\ZZ_{\geq 0}$, define
\[
	l(k):=\begin{cases}
		0 ,& \text{if } k=0\\
		\left\lfloor \frac{\log k}{\log p} \right\rfloor+1, & \text{if } k>0,
	\end{cases}
\]
and
\[
	\chi_k:=(\text{the characteristic function of }k+p^{l(k)}\Zp).
\]
The functions $\chi_k$ are called \emph{wavelets} and form an orthonormal basis for $C(\Zp,\Qp)$.
\begin{prop}[{\cite[Thm. 62.2]{schikhof}}]\label{prop_wavelet_coeff}
	 Every function $f\colon \Zp\rightarrow \Qp$ can be written uniquely as a convergent series
	\[
		f=\sum_{k\geq 0} a_k \chi_k
	\]
	called the \emph{wavelet expansion}. The coefficients $a_k$ in the wavelet expansion are given explicitly by $a_0=f(0)$, and for $k>0$ by the formula
	\[
		a_k=f(k)-f(k_-).
	\]
	Here, $k_-:=\sum_{i=0}^{l(k)-1}b_ip^i$ if $k$ has the $p$-adic expansion $k=\sum_{i=0}^{l(k)}b_ip^i$.
\end{prop}

The wavelet expansion can be used to compute Volkenborn integrals:
\begin{prop}[{\cite[Ex. 62.E.]{schikhof}}]
	If $f\colon\Zp\rightarrow \Qp$ is continuously differentiable with wavelet expansion $f=\sum_{k\geq 0} a_k \chi_k$, then
	\[
		\int_{\Zp}f(t) \mathrm{d}t= \sum_{k\geq 0} a_k p^{-l(k)}.
	\]
\end{prop}
Let us write $C^1(\Zp,\Qp)$ for the space of continuously differentiable functions and $\nu_p(\cdot)$ for the $p$-adic valuation normalized by $\nu_p(p)=1$. Let us define
\[
	\Delta(f):=\inf_{k\geq 0}  \left( \nu_p(a_k)-l(k)\right)
\]
for $f=\sum_{k\geq 0}a_k\chi_k$. With this definition, we have the following immediate corollary:
\begin{cor}\label{cor_Volkint}
	Let $f,g\in C^1(\Zp,\Qp)$ with $\Delta(f-g)>n\in\ZZ$. Then
	\[
		\int_{\Zp}f(t)\mathrm{d}t\equiv \int_{\Zp}g(t)\mathrm{d}t\mod p^{n}\Zp.
	\]
\end{cor}
Let us summarize some basic properties of $\Delta(f)$:
\begin{lem}\label{lem_df}
Let $f,g\in C^1(\Zp,\Qp)$. The function $\Delta$ satisfies the following properties:
\begin{enumerate}
\item  $\Delta(f+g)\geq \min(\Delta(f),\Delta(g))$.
\item $\Delta(f\cdot g)\geq \min(\Delta(f),\Delta(g))$
\item For $l\geq 0$ let $[p^l]\colon \Zp\rightarrow\Zp$ denote multiplication by $p^l$. If $\nu_p(f(0))\geq \Delta(f)+l$, then
\[
	\Delta(f\circ [p^l])\geq \Delta(f)+l.
\]
\item If $f\equiv g \mod p$, then
\[
	\Delta(f^p-g^p)\geq \Delta(f-g)+1.
\]
\item For $m\geq 0$, we have
\[
	\Delta(\binom{x}{m})\geq -l(m).
\]
Here, we write $\binom{x}{m}$ for the function $x\mapsto \binom{x}{m}$.
\item Let $f=\sum_{k\geq 0}{c_k}t^k\in \Qp\llbracket t\rrbracket$ be a power series which converges on $\Zp$, i.e. $\lim_{k\rightarrow\infty} |c_k|_p=0$. Let us view $f$ as a continuous function on $\Zp$ with values in $\Qp$. Then
\[
	\Delta(f)\geq \inf_k \nu_p(c_k).
\]
\end{enumerate}
\end{lem}
\begin{proof}
$(a)$ Let $f=\sum_{k\geq 0}a_k\chi_k$, $g=\sum_{k\geq 0} b_k\chi_k$, then
\begin{align*}
	\Delta(f+g)&:=\inf \{\nu_p(a_k+b_k)-l(k), k\in\ZZ_{\geq 0}\}\\
	&\geq \inf \{\min(\nu_p(a_k),\nu_p(b_k))-l(k), k\in\ZZ_{\geq 0}\}\\
	&\geq \min(\Delta(f),\Delta(g)).
\end{align*}
$(b)$ It is enough to consider the case $f=a_k\chi_k$, $g=b_i\chi_i$ with $i\leq k$ and $a_k,b_i\in \Zp$. Then
\[
	f\cdot g=\begin{cases} a_kb_i \chi_k, & \text{ if }k\equiv i \mod p^{l(i)}\\
		0, & \text{ if }k\not\equiv i \mod p^{l(i)}.  \end{cases}
\]
We deduce
\begin{align*}
	\Delta(f\cdot g)&=\begin{cases} \nu_p(a_k)+\nu_p(b_i)- l(k), & \text{ if }k\equiv i \mod p^{l(i)}\\
		\infty, & \text{ if }k\not\equiv i \mod p^{l(i)}  \end{cases}\\
		&\geq \min(\Delta(f),\Delta(g)).
\end{align*}
$(c)$ We have
\[
	\chi_k\circ[p^l]=\begin{cases} 0, & \text{ if }k\not\equiv 0 \mod p^l\\ \chi_{\frac{k}{p^l}}, & \text{ if }k\equiv 0 \mod p^l.\end{cases}
\]
Thus for $f=\sum_{k\geq 0}a_k\chi_k$, we get the formula
\[
	f\circ[p^l]=\sum_{k\geq 0}a_{p^lk}\chi_{k}.
\]
The definition of $l(\cdot)$ implies, for $k>0$, the property $l(p^lk)=l+l(k)$. The inequality
\[
	\nu_p(a_{p^lk})-l(p^lk)\geq \Delta(f)
\]
shows for $k>0$:
\[
	\nu_p(a_{p^lk})-l(k)\geq \Delta(f)+l.
\]
For $k=0$, we have $\nu_p(a_0)\geq \Delta(f)+l$, by assumption, and the claim follows.\par
$(d)$ Let
\[
	h:=\frac{f-g}{p}\in C(\Zp,\Zp).
\]
Then $\Delta(h)\geq \min(\Delta(\frac{f}{p}),\Delta(\frac{g}{p}))\geq \min(\Delta(f),\Delta(g))-1$. The binomial formula gives
\[
	f^p-g^p=\sum_{i=1}^p \binom{p}{i}p^ih^ig^{p-i}.
\]
Since $\nu_p(\binom{p}{i}p^i)\geq 2$ holds for $1\leq i\leq p$, we get
\[
	\Delta(f^p-g^p)\geq 2+\min_{1\leq i\leq p} \Delta(h^ig^{p-i})\geq 1+\min(\Delta(f),\Delta(g)).
\]
$(e)$ In \cite[Prop. 2.5]{deshalit}, it is shown that
\[
	\binom{x}{m}=\sum_{k\geq 0} b_{m,k}\chi_k
\]
with $\nu_p(b_{m,k})\geq l(k)-l(m)$ and the claim follows.\par
$(f)$ It is enough to prove
\[
	\Delta(ct^n)\geq \nu_p(c).
\]
The formula for the wavelet coefficients, in \Cref{prop_wavelet_coeff}, shows $ct^n=\sum_{k\geq 0} a_k \chi_k(t)$ with $a_k=c\cdot(k^n-(k_-)^n)$. From this, we deduce
\[
	\nu_p(a_k)\geq l(k)+\nu_p(c)
\]
and $\Delta(ct^n)\geq \nu_p(c)$ follows.
\end{proof}

\section{\texorpdfstring{Linear forms in $p$-adic Hurwitz zeta values with related coefficients}{Linear forms in p-adic Hurwitz zeta values with related coefficients}}\label{sec_1}
The aim of this section is the construction of linear forms in $p$-adic $L$-values using Volkenborn integration over certain rational functions.\par

Let us start recalling the definition of the Hurwitz zeta function. For a positive real number $x>0$, the \emph{Hurwitz zeta function} is defined by the convergent series
\[
	\zeta(s,x):=\sum_{n\geq 0}\frac{1}{(x+n)^s},\quad \text{ for } \Re(s)>1,
\]
and can be continued to a holomorphic function on $\CC\setminus\{1\}$ with a simple pole at $s=1$. Its values at the negative integers can be described in terms of the Bernoulli polynomials $B_n(x)$. More precisely, the value at $1-n$ is given by the formula
\[
	\zeta(1-n,x)=-\frac{B_{n}(x)}{n},\quad n\in\ZZ_{\geq 1}, x\in [0,1].
\]
In particular, note that for $n\in\ZZ_{\geq 1}$ and $x\in[0,1]\cap \QQ$ the value $\zeta(1-n,x)$ is rational. For $x\in\QQ$ with large $p$-adic norm, it is possible to interpolate the values of the Hurwitz zeta function at negative integers $p$-adically. For the prime $p$, define
\[
	q_p:=\begin{cases} p, & \text{if } p\neq 2\\ 4, & \text{if }  p=2. \end{cases}
\]
Let us write $\nu_p\colon \Qp\rightarrow \ZZ$ for the $p$-adic valuation with the normalization $\nu_p(p)=1$ and $|x|_p:=p^{-\nu_p(x)}$ for the $p$-adic norm on $\Qp$. The units $\Zp^\times$ of the $p$-adic integers decompose canonically
\[
	\Zp^\times \righteq \mu_{\phi(q_p)}(\Zp)\times (1+q_p\Zp).
\]
Here, $\mu_n(R)$ denotes the group of $n$-th roots of unity in a ring $R$ and $\phi(n)$ denotes Euler's totient function. The canonical projection
\[
	\omega \colon \Zp^\times \rightarrow \mu_{\phi(q_p)}(\Zp)
\]
is called the Teichm\"uller character. Let us extend the Teichm\"uller character to a map
\[
	\Qp^\times\rightarrow \Qp^\times,
\]
 by setting
\[
	\omega(x):=p^{\nu_p(x)}\omega(x/p^{\nu_p(x)}),
\]
and define $\langle x\rangle:=\frac{x}{\omega(x)}$ for $x\in\Qp^\times$. Let $x\in\Qp$ be given with $|x|_p\geq q_p$. It can be shown that there is a unique $p$-adic meromorphic function $\zeta_p(s,x)$ on
\[
	\{s\in\Cp\mid |s|_p<q_p p^{-1/(p-1)}\}
\]
such that
\[
	\zeta_p(1-n,x)=-\omega(x)^{-n}\frac{B_n(x)}{n},
\]
see \cite[\S 11.2]{cohen2}. Explicitly, the function $\zeta_p(s,x)$ can be defined by a Volkenborn integral, see \cite[Def. 11.2.5]{cohen2}:
\begin{equation}\label{eq_zeta_Volkenborn}
	\zeta_p(s,x)=\frac{1}{s-1}\int_{\Zp}\langle x+t\rangle^{1-s}\mathrm{d}t.
\end{equation}
An important source of linear forms in values of the classical Riemann zeta function is provided by summation of rational functions of degree less than $-1$, having only poles of maximal order $s$, over the positive integers. The following lemma together with the partial fraction decomposition of rational functions suggests that Volkenborn integration over rational functions is a good substitute for summation in the $p$-adic case.
\begin{lem}\label{lem_zeta_Volkenborn}
	Let $s>1$ be an integer and $x\in\Qp$ with $|x|_p \geq q_p$, then
	\[
		\omega(x)^{1-s}\zeta_p(s,x)=\frac{1}{s-1}\int_{\Zp} (x+t)^{1-s} \mathrm{d}t.
	\]
\end{lem}
\begin{proof}
For $m\in\ZZ$ and $x\in\Qp$ with $|x|_p\geq q_p$, we get $\omega(x)=\omega(x+m)$. Unravelling the definition of the Volkenborn integral in formula \eqref{eq_zeta_Volkenborn} gives
\[
	\zeta_p(s,x)=\frac{1}{s-1}\lim_{r\rightarrow\infty} p^{-r}\sum_{0\leq m<p^r} \langle x+m \rangle^{1-s},
\]
and we conclude
\begin{align*}
	\omega(x)^{1-s}\zeta_p(s,x)&=\omega(x)^{1-s}\frac{1}{s-1}\lim_{r\rightarrow\infty} p^{-r}\sum_{0\leq m<p^r} \langle x+m \rangle^{1-s}\\
	&=\frac{1}{s-1}\lim_{r\rightarrow \infty}p^{-r}\sum_{0\leq m<p^r} (x+m)^{1-s}.
\end{align*}
\end{proof}
For a Dirichlet character $\chi$ define $\delta=\delta(\chi)$ depending on the parity, i.e.
\[
	\delta:=\begin{cases} 0, & \text{if } \chi(-1)=1\\
	1, & \text{if } \chi(-1)=-1. \end{cases}
\]
The following lemma guarantees the non-vanishing of half of the generalized Bernoulli numbers:
\begin{lem}\label{lem_binomial_nontrivial}
	Let $k$ be a positive integer with $k\equiv \delta\mod 2$. Then $B_{k,\chi}\neq 0$.
\end{lem}
\begin{proof}
This follows, for example, from the formula $L(1-k,\chi)=-\frac{B_{k,\chi}}{k}$ and the non-vanishing of the corresponding $L$-value.
%Let us recall the formula
%\[
%	L(1-k,\chi)=-\frac{B_{k,\chi}}{k}.
%\]
%Thus, it is enough to prove $L(1-k,\chi)\neq 0$ for $k\equiv\delta\mod 2$. Let $X$ be the group of Dirichlet characters generated by $\chi$ and let $F$ be the abelian field extension corresponding to $X$. The Dedekind zeta function $\zeta_F$ decomposes as a product of Dirichlet $L$-functions
%\[
%	\zeta_F(s)=\prod_{\tilde{\chi}\in X}L(s,\tilde{\chi}).
%\]
%Now, the claim follows from the functional equation of $\zeta_F$ and the non-vanishing of $\zeta_F$ at positive integers.
\end{proof}

\begin{defin}
 For non-negative integers $m,n,m_1,\dots,m_k$ with $m\geq m_1+\dots+m_k$,
 \[
 	\binom{m}{m_1,\dots,m_k}:=\binom{m}{m_1}\binom{m-m_1}{m_2}\dots \binom{m-m_1-\dots-m_{k-1}}{m_k}
 \]
 denotes the \emph{multinomial coefficient}. It will be convenient to define 
 \[
 	\binom{m}{\underline{n}}:=\binom{m}{\underbrace{n,\dots,n}_{\left\lfloor\frac{m}{n}\right\rfloor\text{-times}}}.
 \]
 So $\binom{m}{\underline{n}}$ is the multinomial coefficients with as many $n$s as possible in the denominator.
\end{defin}

In the following, we fix the notation which will be used throughout the rest of the paper:
\begin{notation}\label{main_notation}
Let $\chi$ be a Dirichlet character of conductor
\begin{equation}
	d=d'p^{l_0},\quad \text{with } \gcd(d',p)=1.
\end{equation}
By \Cref{lem_binomial_nontrivial}, the Bernoulli number $B_{2+\delta,\chi}$ is non-zero, so
\begin{equation}\label{def_r}
	r:=\left\lfloor\nu_p\left(B_{2+\delta,\chi}\right)\right\rfloor+1,
\end{equation}
is a well-defined integer. Let $l\geq l_0$ be a positive integer and set 
\begin{equation}\label{def_QD}
Q=Q(l):=p^{r+l+1}, \text{ as well as }D=D(l):=d'p^l.
\end{equation}
The choice of $l$ will be specified later. For $n\in \ZZ_{>0}$, let us define
\begin{equation}\label{def_N}
	N(n):=p^l(p^{\lfloor \log_p(d'n) \rfloor+1}-1).
\end{equation}
Thus $N(n)$ is the least integer $\geq Dn$ of the form $p^l(p^m-1)$ for suitable $m\in\ZZ$. For positive integers $s,n$ with $s\geq pQD$, let us define
\begin{equation}\label{eq_def_Rn}
	R_{n}(t):=n!^s \binom{N(n)}{\underline{n}}^Q\binom{Dt+N(n)}{N(n)}^Q\frac{(Dt)^{2+\delta}}{(t)_{n+1}^s}.
\end{equation}
Here, we write $(t)_n=t(t+1)\dots(t+n-1)$ for the rising factorial.
\end{notation}

Of course, the definition of $R_n(t)$ does also depend on $s$, $l$ and $\chi$, but later these parameters will be fixed while $n$ tends to infinity. The partial fraction decomposition of $R_n(t)$ with respect to $t$ gives
	\[
		R_n(t)=\sum_{i=1}^{s}\sum_{k=0}^{n}\frac{r_{i,k}}{(t+k)^i},
	\]
	with
	\[
	r_{i,k}=\frac{1}{(s-i)!}\left . \left( \frac{\partial}{\partial t} \right)^{s-i} \left [ R_n(t)(t+k)^s\right ]\right |_{t=-k}.
	\]
	Using the coefficients in the partial fraction decomposition of $R_n$, we define
	\begin{align}
		\rho_{0,x}^{(n)}:&=-\sum_{i=1}^{s}\sum_{k=0}^{n}\sum_{\nu=0}^{k-1} i\cdot r_{i,k}\left(\nu+x\right)^{-i-1},\quad x\in\Qp\setminus\ZZ_{<0} \label{eq_def_rho1} \\
		\rho_i^{(n)}:&=\sum_{k=0}^n i\cdot r_{i,k},\quad 1\leq i\leq s. \label{eq_def_rho2}
	\end{align}
Note that $\rho_i^{(n)}$, for $1\leq i\leq s$, is a rational number which is independent of $x$. If $x\in\QQ\setminus\ZZ_{<0}$, then $\rho_{0,x}^{(n)}$ is also a rational number.
\begin{prop}\label{prop_linearforms}
Let $R_n$ be the rational function defined in \eqref{eq_def_Rn}. If $x\in\Qp$ with $|x|_p\geq p$ (resp. $|x|_2\geq 4$ if $p=2$), then
	\[
		\int_{\Zp} R_n(t+x) \mathrm{d}t=\rho^{(n)}_{0,x}+\sum_{i=1}^s \rho_i^{(n)}\omega(x)^{-i}\zeta_p\left(i+1,x\right).
	\]
	Here, $\rho_{0,x}^{(n)}$ and $\rho_{i}^{(n)}$ are the $p$-adic numbers defined in \eqref{eq_def_rho1} and \eqref{eq_def_rho2}. 
\end{prop}
\begin{proof}
Let us first observe that condition $|x|_p\geq p$ (resp. $|x|_2\geq 4$ if $p=2$) implies that $R_n(t+x)$ is a rational function without poles in $\Zp$. In particular, $R_n(t+x)$ is continuously differentiable on $\Zp$. 
	This allows us to compute the Volkenborn integral as follows:
	\begin{align*}
		\int_{\Zp} R_n(t+x) \mathrm{d}t&=\int_{\Zp}\sum_{i=1}^{s}\sum_{k=0}^{n}\frac{r_{i,k}}{\left(t+x+k\right)^i}  \mathrm{d}t\\
		&=\sum_{i=1}^{s}\sum_{k=0}^{n}r_{i,k}\int_{\Zp}\left(t+x+k\right)^{-i}  \mathrm{d}t\\
		&\stackrel{(\text{Prop.} \ref{prop_volkenborn_transl})}{=}\sum_{i=1}^{s}\sum_{k=0}^{n}r_{i,k}\left[\int_{\Zp}\left(t+x\right)^{-i}  \mathrm{d}t -i\sum_{\nu=0}^{k-1} \left(\nu+x\right)^{-i-1}  \right]\\
		&\stackrel{(\text{Lem.} \ref{lem_zeta_Volkenborn})}{=}-\sum_{i=1}^{s}\sum_{k=0}^{n}\sum_{\nu=0}^{k-1} i\cdot r_{i,k}\left(\nu+x\right)^{-i-1} + \sum_{i=1}^{s}\sum_{k=0}^{n}i\cdot r_{i,k}\omega(x)^{-i}\zeta_p\left(i+1,x  \right)\\
		&= \rho_{0,x}^{(n)}+\sum_{i=1}^s \rho^{(n)}_i\omega(x)^{-i}\zeta_p\left(i+1,x\right).
	\end{align*}
\end{proof}

The $p$-adic $L$-values of $\chi$ can be expressed using the $p$-adic Hurwitz zeta values. For $i\in\ZZ\setminus\{1\}$, we have the formula, cf. \cite[Def. 11.3.4]{cohen2}:
\begin{equation}\label{eq_Hurwitz_L}
	L_p(i,\chi)=\frac{\langle D \rangle^{1-i}}{D} \sum_{\substack{1\leq j\leq D \\ \gcd(j,p)=1 }} \chi(j)\zeta_p\left(i,\frac{j}{D}\right).
\end{equation}
For a positive integer $n$, let $\dn{n}$ be the least common multiple of $1,\dots,n$. Let $K\subseteq\bar{\QQ}$ be a number field containing the field of definition $\QQ(\chi)$ of the Dirichlet character $\chi$. We define the linear form
\begin{equation}\label{def_Lambda}
	\Lambda_n(X_0,\dots,X_{s}):=\lambda_{0,n}X_0+\dots+\lambda_{s,n}X_s,\quad \lambda_{i,n}\in K,
\end{equation}
with
\begin{align*}
	\lambda_{0,n}:&=(s-1)!\dn{n}^{s-1} \sum_{\substack{j=1\\ \gcd(j,p)=1}}^{D}\chi(j) \rho_{0,\frac{j}{D}}^{(n)},\\
	\lambda_{i,n}:&=(s-1)!\dn{n}^{s-1} D^{i+1} \rho_i^{(n)}, \quad \text{ for }1\leq i \leq s.
\end{align*}
As an immediate consequence of \Cref{prop_linearforms}, we get:
\begin{equation}\label{cor_L_linforms}
	(s-1)!\dn{n}^{s-1}\sum_{\substack{1\leq j\leq D \\ \gcd(j,p)=1}} \chi(j)\int_{\Zp} R_n(t+\frac{j}{D}) \mathrm{d}t=\Lambda_n\left(1,L_p(2,\chi\omega^{-1}),\dots,L_p(s+1,\chi\omega^{s})\right)
\end{equation}
This gives the desired linear forms in $p$-adic $L$-values, which will allow us later to deduce the linear independence result. The factor $(s-1)!\dn{n}^{s-1}$ is introduced in order to make the coefficients integral, see \Cref{cor_integrality}.

\section{Arithmetic properties of the linear forms}\label{sec_1_1}
Let us keep the setup introduced in \Cref{main_notation}. In the following, we study the arithmetic properties of the coefficients $\rho_{0,x}^{(n)}$ and $\rho_{i}^{(n)}$, defined in \eqref{eq_def_rho1} and \eqref{eq_def_rho2}. It is well-known that the binomial polynomials form a basis for all integer-valued polynomials. The following auxiliary lemma will be helpful for proving the integrality of the linear forms $\Lambda_n$:
\begin{lem}\label{lem_integral_binom}
	Let $\lambda,k,n$ and $m$ be non-negative integers satisfying $1\leq k\leq m$.
	\begin{enumerate}
	\item $d_n^\lambda \frac{\partial^\lambda}{\partial t^\lambda}\binom{m}{\underline{n}}\binom{t}{m}\in\QQ[t]$ is an integer-valued polynomial.
	\item $d_n^\lambda \frac{\partial^\lambda}{\partial t^\lambda} \frac{1}{t-k+1} \binom{m}{\underline{n}}\binom{t}{m}\in\QQ[t]$ is an integer-valued polynomial.
	\end{enumerate}
\end{lem}
\begin{proof}
It is well-known that (cf. \cite[Ch. IX]{cahen_chabert})
\begin{equation}\label{eq_integral_binom_1}
	d_n \cdot \frac{\partial}{\partial t} \binom{t}{k} \text{ is integer-valued for all }k\leq n.
\end{equation}
$(a)$ If $m<n$ the claim follows immediately from \eqref{eq_integral_binom_1}, thus we may assume $m\geq n$. The formula
\begin{equation}\label{eq_integral_binom_1a}
	\binom{m}{\underline{n}}\binom{t}{m}=\binom{t}{n}\binom{t-n}{n}\cdot \dots \cdot \binom{t-n\left\lfloor \frac{m}{n}\right\rfloor}{m-n\left\lfloor \frac{m}{n}\right\rfloor}
\end{equation}
shows that $\binom{m}{\underline{n}}\binom{t}{m}$ decomposes as a product of integer-valued polynomials of degree $\leq n$. Now the claim follows from \eqref{eq_integral_binom_1} by applying the Leibniz rule to \eqref{eq_integral_binom_1a} and observing that every integer-valued polynomial of degree $N$ can be written as an integral linear combination of binomial polynomials of degree $\leq N$.\par 
$(b)$ By the definition of the multinomial coefficients, we have
\begin{align}
	\label{eq_integral_binom_2} \binom{m}{\underline{n}}&=\frac{m!}{(m-k)!(k-1)!}\binom{m-k}{\underline{n}}\binom{k-1}{\underline{n}}\binom{m-n\left( \left\lfloor\frac{m-k}{n}\right\rfloor+\left\lfloor\frac{k-1}{n}\right\rfloor \right)}{\underline{n}},
\end{align}
and
\begin{equation}\label{eq_integral_binom_3}
	\frac{1}{t-k+1}\binom{t}{m}=\frac{(k-1)!(m-k)!}{m!}\binom{t}{k-1}\binom{t-k}{m-k}.
\end{equation}
Combining \eqref{eq_integral_binom_2} with \eqref{eq_integral_binom_3} gives:
\[
	\frac{1}{t-k+1}\binom{m}{\underline{n}}\binom{t}{m}=\binom{t}{k-1}\binom{k-1}{\underline{n}}\binom{t-k}{m-k}\binom{m-k}{\underline{n}}\binom{m-n\left( \left\lfloor\frac{m-k}{n}\right\rfloor+\left\lfloor\frac{k-1}{n}\right\rfloor \right)}{\underline{n}}.
\]
Now the result follows from $(a)$ using the Leibniz rule.
\end{proof}

\begin{prop}\label{prop_arith} Let $x=\frac{j}{D}$ with $1\leq j\leq D$ and $\gcd(j,p)=1$.
\begin{enumerate}
\item For $1\leq i\leq s$, we have
\begin{equation}\label{arith_eq4}
	(s-i)!\dn{n}^{s-i} \rho_i^{(n)} \in \ZZ
\end{equation}
\item For $i=0$, we get the formula
\begin{equation}\label{arith_eq3}
	(s-1)!\dn{n}^{s-1} \rho_{0,x}^{(n)}\in \ZZ.
\end{equation}
\end{enumerate}
\end{prop}
\begin{proof} 
Let us first note that the partial fraction decomposition
\[
	\frac{n!}{(t)_{n+1}}=\sum_{m=0}^n\frac{(-1)^m\binom{n}{m}}{t+m}
\]
implies, for $0\leq \lambda\leq s$ and $0\leq k\leq n$, the formula:
\begin{equation}\label{eq_partial_fraction1}
	d_n^\lambda\left.\frac{\partial^\lambda}{\partial t^\lambda}\frac{n!}{(t)_{n+1}}\right|_{t=-k}\in \ZZ.
\end{equation}
$(a)$ In formula \eqref{eq_def_rho2}, we have defined
\begin{align*}
		\rho_i:=\sum_{k=0}^n i\cdot r_{i,k},\quad 1\leq i\leq s,
\end{align*}
where $r_{i,k}$ are the coefficients in the partial fraction decomposition of $R_n$:
\[
	r_{i,k}=\frac{1}{(s-i)!}\left . \left( \frac{\partial}{\partial t} \right)^{s-i} \left [  \binom{N(n)}{\underline{n}}^Q\binom{Dt+N(n)}{N(n)}^Q\left(\frac{n!}{(t)_{n+1}}\right)^s(Dt)^{2+\delta}(t+k)^s\right ]\right |_{t=-k}.
\]
%The polynomials $(Dt)^{2+\delta}(t+k)^s$ are integral and the formula
%\begin{equation}\label{eq_binom_der1}
%	d_n^\lambda\left.\frac{\partial^\lambda}{\partial t^\lambda}\binom{N(n)}{\underline{n}}\binom{Dt+N(n)}{N(n)}\right|_{t=-k}\in \ZZ
%\end{equation}
%follows from \Cref{lem_integral_binom}.
Applying the Leibniz rule and observing \Cref{lem_integral_binom}, as well as \eqref{eq_partial_fraction1}, proves $(a)$.
\par 
$(b)$ We proceed similarly as in $(a)$. We have, according to \eqref{eq_def_rho1}, the formula
\[
	\rho_{0,x}^{(n)}:=-\sum_{i=1}^{s}\sum_{k=0}^{n}\sum_{\nu=0}^{k-1} i\cdot r_{i,k}\left(\nu+x\right)^{-i-1}
\]
with
	\[
	r_{i,k}=\frac{1}{(s-i)!}\left . \left( \frac{\partial}{\partial t} \right)^{s-i} \left [ R_n(t)(t+k)^s\right ]\right |_{t=-k}.
	\]
We compute for $k\in\{0,\dots,n\}$:
%\begin{align*}
%	\sum_{i=1}^{s}r_{i,k}\sum_{\nu=0}^{k-1} i\cdot \left(\nu+x\right)^{-i-1}=&\sum_{i=1}^s \frac{1}{(s-i)!}\left.\left(\frac{\partial}{\partial t}\right)^{s-i}\left[ R_n(t)(k+t)^s \right]\right|_{t=-k}\\
%	& \quad \times \frac{1}{(i-1)!}\left.\left( \frac{\partial}{\partial t} \right)^{i-1}\left[ \sum_{\nu=1}^{k} \left(-t-\nu+x\right)^{-2}  \right]\right|_{t=-k}\\
%	=&\frac{1}{(s-1)!}\left.\left(\frac{\partial}{\partial t}\right)^{s-1}\left[R_n(t)(k+t)^s\sum_{\nu=1}^{k} \left(t+\nu-x\right)^{-2} \right]\right|_{t=-k}
%\end{align*}
\begin{align*}
	\sum_{i=1}^{s}r_{i,k}\sum_{\nu=0}^{k-1} i\cdot \left(\nu+x\right)^{-i-1}=\frac{1}{(s-1)!}\left.\left(\frac{\partial}{\partial t}\right)^{s-1}\left[R_n(t)(k+t)^s\sum_{\nu=1}^{k} \left(t+\nu-x\right)^{-2} \right]\right|_{t=-k}.
\end{align*}
This gives
\begin{equation}\label{eq_rhoxn_derivative}
	\rho_{0,x}^{(n)}=\sum_{k=0}^n\frac{1}{(s-1)!}\left(\frac{\partial}{\partial t}\right)^{s-1}\left.\left[ R_n(t)(k+t)^s\sum_{\nu=1}^k  \frac{1}{(t+\nu-x)^2}  \right]\right|_{t=-k}.
\end{equation}
For $1\leq\nu\leq k\leq n$, we have the formula
\begin{multline*}
	R_n(t)(k+t)^s\frac{1}{(t+\nu-x)^2}=\\
	=\frac{D^2}{(Dt+N(n)-(N(n)-D\nu+Dx))^2}\binom{N(n)}{\underline{n}}^Q\binom{Dt+N(n)}{N(n)}^Q\left(\frac{n!}{(t)_{n+1}}\right)^s(Dt)^{2+\delta}(t+k)^s
\end{multline*}
%The bounds $1\leq \nu\leq n$, $1\leq Dx\leq D$ and $N(n)>Dn$ show that $N(n)-D\nu+Dx$ is an integer satisfying $0\leq N(n)-D\nu+Dx< N(n)$.
Applying \Cref{lem_integral_binom} yields, for $0\leq \lambda\leq s-1$:
\begin{equation}\label{eq_integrality1}
	d_n^{\lambda}\left(\frac{\partial}{\partial t}\right)^{\lambda}\left.\frac{D}{(Dt+N(n)-(N(n)-D\nu+Dx))}\binom{N(n)}{\underline{n}}\binom{Dt+N(n)}{N(n)}\right|_{s=-k}\in\ZZ,
\end{equation}
and we conclude again using the Leibniz rule.
%
%
%\begin{align*}
%	F_\alpha(t):&=\frac{(Dt+D\alpha+1)_{n}}{(t)_{n+1}}(k+t),\quad G(t):=\frac{n!}{(t)_{n+1}}(k+t)\\
%	H(t):&=(Dt)^{2+\delta}\prod_{i=Dn+1}^{N(n)}(Dt+i)^Q
%\end{align*}
%Then $D(\nu-x)$ is an integer satisfying $1\leq D(\nu-x)\leq Dn$ and thus $t+\nu-x$ divides one of the numerators of the $F_\alpha$. Let $0\leq \tilde{l}\leq D-1$ be the index such that $t+\nu-x$ divides the numerator of $F_{\frac{\tilde{l}n}{D}}(t)$. Then
%\[
%	\frac{1}{(t+\nu-x)}F_{\frac{\tilde{l}n}{D}}(t)=\sum_{\substack{m=0\\ m\neq k}}^n \frac{(k-m)(-1)^m\tilde{f}_m}{m+t}
%\]
%with
%\[
%	\tilde{f}_m=\frac{D}{D\nu-Dm-Dx}\binom{n}{m}\binom{D(\alpha-m)+n}{n}.
%\]
%Since $D\nu-Dm-Dx\in\ZZ $ and $0<|D(\nu-m-x)|\leq Dn$, we deduce $\dn{Dn}\tilde{f}_m\in\ZZ$. For $\lambda\geq 1$ we get
%\begin{equation}\label{eq_2}
%	\dn{n}^\lambda \dn{Dn} \left. \partial^{(\lambda)} \left( \frac{1}{(t+\nu-x)}F_{\frac{\tilde{l}n}{D}}(t) \right)\right|_{t=-k}= \dn{n}^\lambda \dn{Dn}(-1)^{\lambda+1} \sum_{\substack{m=0 \\ m\neq k}}^n \frac{(-1)^m\tilde{f}_m}{(m-k)^\lambda}\in \ZZ
%\end{equation}
%and we conclude again by applying the Leibniz formula to
%\[
%	R_n(t)(k+t)^s  \frac{1}{(t+\nu-x)^2}={H}(t){G}(t)^{s-2D}\left( \prod_{\substack{l=0 \\ l\neq \tilde{l}}}^{D-1}F_{\frac{ln}{D}}(t)^Q\right)\cdot F_{\frac{\tilde{l}n}{D}}(t)^{Q-2} \cdot\left(\frac{F_{\frac{\tilde{l}n}{D}}(t)}{(t+\nu-x)}\right)^2.
%\]
\end{proof}
%
%In \eqref{cor_L_linforms} we have shown the existence of linear forms $\Lambda_n(X_0,\dots,X_s)$ satisfying
%\[
%	(s-1)!\dn{n}^{s-1} \sum_{\substack{1\leq j\leq D \\ \gcd(j,p)=1}} \chi(j)\int_{\Zp} R_n(t+\frac{j}{D}) \mathrm{d}t=\Lambda_n\left(1,(L_p(i,\chi\omega^{1-i}))_{i=2}^{s+1}\right).
%\]
\Cref{prop_arith} allows us to deduce the desired integrality properties for the linear forms $\Lambda_n$, which have been defined in \eqref{def_Lambda}:
\begin{cor}\label{cor_integrality}
Let $K$ be a finite extension of $\QQ(\chi)$. The linear forms $(\Lambda_n)_{n}$ have integral coefficients, i.e.
	\[
		\Lambda_n(X_0,\dots,X_s)\in \Ocal_{K} X_0+\dots+ \Ocal_{K} X_s.
	\]
\end{cor}

\section{Asymptotic growth of the coefficients}\label{sec_1_2}
The aim of this section is to bound the asymptotic growth of the coefficients of the linear forms $\Lambda_n$ in the Archimedean norm. Let us keep the setup, introduced in \Cref{main_notation}. %and let $j$ be an integer between $0$ and $D$. 
%Recall the formula
%\[
%		\int_{\Zp} R_n(t+j/D) \mathrm{d}t=\rho^{(n)}_{0,j/D}+\sum_{i=1}^s \rho_i^{(n)}\omega(j/D)^{-i}\zeta_p\left(i+1,j/D\right).
%\]

\begin{prop}\label{prop_growth}For $0\leq j\leq D$, we have the following estimates
 \[
 	\limsup_{n\rightarrow\infty} |\rho^{(n)}_{0,j/D}|^{1/n}\leq (pD)^{pQD}2^{s}\quad \text{and } \limsup_{n\rightarrow\infty} |\rho_i^{(n)}|^{1/n}\leq (pD)^{pQD}2^{s},
 \]
 for the rational numbers $\rho_i^{(n)}$ and $\rho^{(n)}_{0,j/D}$, which have been introduced in \eqref{eq_def_rho1} and \eqref{eq_def_rho2}.
\end{prop}
\begin{proof} Let us write $x:=j/D$. By the definition of $\rho_i^{(n)}$ and $\rho_{0,x}^{(n)}$, we deduce
\[
	|\rho_i^{(n)}|\leq \sum_{k=0}^n  i |r_{i,k}|,\quad |\rho_{0,x}^{(n)}|\leq \sum_{k=0}^{n}\sum_{i=1}^{s}\sum_{\nu=0}^{k-1} i\cdot \left|\nu+x\right|^{-i-1}|r_{i,k}|
\]
and it suffices to prove
\[
	\limsup_{n\rightarrow\infty} |r_{i,k}|^{1/n}\leq (pD)^{pQD}2^{s}.
\]
By Cauchy's integral formula, we have
\begin{align*}
	r_{i,k}= \frac{1}{2\pi i}\int_{|t+k|=\frac{1}{D}} \binom{N(n)}{\underline{n}}^Q\binom{Dt+N(n)}{N(n)}^Q \frac{ n!^{s}}{(t)_{n+1}^s} (Dt)^{2+\delta}(t+k)^{i-1}\mathrm{d}t,
\end{align*}
and hence 
\begin{equation}\label{eq_growth_0}
	|r_{i,k}|\leq \sup_{|t+k|=\frac{1}{D}}\left| \binom{N(n)}{\underline{n}}^Q\binom{Dt+N(n)}{N(n)}^Q \frac{ n!^{s}}{(t)_{n+1}^s} (Dt)^{2+\delta} \right|.
\end{equation}
In the following, we estimate each of the terms in \eqref{eq_growth_0}. The multinomial formula
%\[
%	(\underbrace{1+\dots+1}_{\left\lceil \frac{N(n)}{n}\right\rceil})^{N(n)}=\sum_{i_1+\dots+i_{\left\lceil \frac{N(n)}{n}\right\rceil}=N(n)}\binom{N(n)}{i_1,\dots,i_{\left\lceil \frac{N(n)}{n}\right\rceil}}
%\]
shows the estimate
\begin{equation}\label{eq_growth_1}
	\binom{N(n)}{\underline{n}}\leq \left\lceil \frac{N(n)}{n}\right\rceil^{N(n)}\leq (pD)^{pDn}.
\end{equation}
%Let us now estimate the term 
%	$$\left|\binom{Dt+N(n)}{N(n)}\right|$$
%for $|t+k|\leq \frac{1}{D}$. 
For $k\geq 0$ and $|t+k|\leq \frac{1}{D}$, we have 
%\begin{equation}\label{eq_growth_2}
%	\left| \binom{Dt+N(n)}{N(n)} \right|\leq 1,
%\end{equation}
%and for $k=0$ and $|t+k|\leq \frac{1}{D}$ we still have
\begin{equation}\label{eq_growth_3}
	\left| \binom{Dt+N(n)}{N(n)} \right|\leq \binom{1+N(n)}{N(n)}=1+N(n)\leq 1+ pDn.
\end{equation}
Next, let us consider $\frac{n!}{|(t)_{n+1}|}$. For $t\in\CC$ with $|t+k|=\frac{1}{D}$, we get
\begin{align*}
	|(t)_{n+1}|\geq \prod_{\iota=0}^{n}\left| -\frac{1}{D}+|\iota-k| \right|
\end{align*}
If $|\iota-k|>1$, we have $\left| -\frac{1}{D}+|\iota-k| \right|\geq |\iota-k|-1$, otherwise we get $\left| -\frac{1}{D}+|\iota-k| \right|\geq \frac{1}{D}$. We obtain the estimate
%\begin{equation*}
%	|(t)_{n+1}|\geq \frac{1}{D^3}k!(n-k)!
%\end{equation*}
%and finally the estimate
\begin{equation}\label{eq_growth_4}
	\frac{n!}{|(t)_{n+1}|} \leq D^3 \binom{n}{k}\leq D^3 \cdot 2^n.
\end{equation}
The term $(Dt)^{2+\delta}$ can be estimated by
\begin{equation}\label{eq_growth_5}
	|(Dt)^{2+\delta}|\leq (Dn)^3.
\end{equation}
Combining \eqref{eq_growth_1}, \eqref{eq_growth_3}, \eqref{eq_growth_4} and \eqref{eq_growth_5} shows
\[
	|r_{i,k}|\leq (pD)^{pDQn}2^{sn}(1+pDn)^Q D^{3s+3}n^3,
\]
and we deduce
\[
	\limsup_n |r_{i,k}|^{\frac{1}{n}}\leq (pD)^{pQD}2^s
\]
as desired.
\end{proof}
\Cref{prop_growth} allows us to bound asymptotically the Archimedean heights of the linear forms $\Lambda_n$, introduced in \eqref{def_Lambda}:
\begin{cor}\label{cor_growth} Let $K$ be a finite extension of $\QQ(\chi)$. The height $H_K(\Lambda_n)$ of the family $(\Lambda_n)_n$ of linear forms can be estimated by the formula
\[
		H_K(\Lambda_n)\leq \exp\left( n\left[\tau_\infty(l,s)+o(1)\right] \right) \text{ as } n\rightarrow \infty,
\]
with
	\[
		\tau_{\infty}(l,s)=[K:\QQ]\left( s\log 2+d'p^{r+2+2l}\log(d'p^{1+l})+s\right).
	\]
	Here, $o(1)$ stands for any (non-negative) sequence tending to $0$ as $n\rightarrow\infty$.
\end{cor}
\begin{proof}
Let us first recall that the height is defined by the formula
\[
	H_K(\Lambda_n):=\max_{0\leq i\leq s} |N_K(\lambda_{i,n})|.
\]
The norms of the coefficients can be estimated by
\[
	|N_K((s-1)!\dn{n}^{s-1} \rho_i^{(n)})|=|(s-1)!\dn{n}^{s-1} \rho_i|^{[K:\QQ]}
\]
and
\begin{align*}
	\left|N_K\left((s-1)!\dn{n}^{s-1}\sum_{1\leq j\leq D} \chi(j)\rho_{0,\frac{j}{D}}^{(n)} \right)\right|\leq \left( \sum_{1\leq j\leq D} \left|(s-1)!\dn{n}^{s-1}\rho_{0,\frac{j}{D}}^{(n)}\right| \right)^{[K:\QQ]}.
\end{align*}
According to the prime number theorem, we have
\[
	\lim_{n\rightarrow \infty}\frac{1}{n}\log |\dn{n}|=1,
\]
and \Cref{prop_growth} gives the estimate
 \[
 	 \frac{1}{n}\log H_K(\Lambda_n)\leq [K:\QQ]\left[\log(2)s+pQD\log(pD) +s\right]+o(1) \text{ as }n\rightarrow\infty.
 \]
Substituting $Q:=p^{r+l+1}$ and $D:=d'p^l$ concludes the proof of the corollary.
\end{proof}

\section{\texorpdfstring{$p$-adic convergence of the linear forms}{p-adic convergence of the linear forms}}\label{sec_1_3}
In this section, we control the $p$-adic convergence of the linear forms
\[
\Lambda_n\left(1,L_p(2,\chi\omega^{-1}),\dots,L_p(s+1,\chi\omega^{s})\right).
\]
Let us first discuss a few auxiliary results about binomial coefficients and Volkenborn integration. We start with the following 'classical' results due to Kummer and Lucas:
\begin{lem}[{\cite[p. 116]{kummer}}]
	The $p$-adic valuation of the binomial coefficient $\binom{n}{k}$ is given by the number of 'carry overs' when adding $k$ and $n-k$ in the $p$-adic base.
\end{lem}

\begin{lem}[{\cite{lucas}}]\label{lem_lucas}
	For non-negative integers $m$ and $n$, we have
	\[
		\binom{m}{n}\equiv \prod_{i=0}^r \binom{m_i}{n_i}\mod p
	\]
	where $m=\sum_{i=0}^r p^i m_i$ and $n=\sum_{i=0}^r p^i n_i$ are the $p$-adic expansions of $n$ and $m$.
\end{lem}
Let us keep the previous notation, introduced in \Cref{main_notation}. For $j\in\ZZ_{\geq 0}$, define the characteristic function $\mathcal{X}_{j,n}\colon \Zp\rightarrow \{0,1\}$ by
\[
	\mathcal{X}_{j,n}(x):=\begin{cases} 1, & \text{if } x\equiv -(d')^{-1}\lfloor j/p^l\rfloor \mod p^{\lfloor \log_p(nd') \rfloor+1} \\ 0, &  \text{else.} \end{cases}
\]
The results of Kummer and Lucas imply the following corollary:
\begin{cor}\label{cor_chi}
	 Let $j\in\ZZ_{\geq 0}$, $n\in\ZZ_{>0}$ and $x\in\Zp$. Then
	\[
		\binom{N(n)+Dx+j}{N(n)}\equiv \mathcal{X}_{j,n}(x)  \mod p.
	\]
\end{cor}

\begin{proof} Since $\ZZ_{\geq 0}$ is dense in $\Zp$ and both functions are $p$-adically continuous, it is enough to prove the claim for $x\in\ZZ_{\geq 0}$. By the definition of $N(n)$, we have $N(n)=p^l(p^m-1)$, for $m:=\lfloor \log_p(nd') \rfloor+1$. Let us write
\[
		Dx+j=p^{l}\left[ d'x+\left\lfloor\frac{j}{p^{l}}\right\rfloor \right] + \underbrace{j-p^{l}\left\lfloor \frac{j}{p^{l}}  \right\rfloor}_{<p^{l}}.
	\]
Since $p^l(p^m-1)=\sum_{i=l}^{m+l-1}(p-1)p^i$, there are no 'carry overs' in the $p$-adic addition of $p^l(p^m-1)$ and $Dx+j$ if and only if $d'x+\lfloor\frac{j}{p^{l}}\rfloor\equiv 0\mod p^m$. If $x\not\equiv -(d')^{-1}\lfloor j/p^l\rfloor \mod p^{m}$, this proves
\[
	\binom{N(n)+Dx+j}{N(n)}\equiv 0\mod p.
\]
Let us now assume $x\equiv -(d')^{-1}\lfloor j/p^l\rfloor \mod p^{m}$. In this case, the integer $Dx+p^l\lfloor\frac{j}{p^{l}}\rfloor$ is divisible by $p^{l+m}$, say $Dx+\lfloor\frac{j}{p^{l}}\rfloor=Ap^{l+m}$, for some positive integer $A$. This allows us to express $N(n)+Dx+j$ in the form
\[
	N(n)+Dx+j=p^l(p^m-1)+Dx+j=A\cdot p^{m+l}+p^l(p^m-1)+\underbrace{j-p^{l}\lfloor \frac{j}{p^{l}} \rfloor}_{<p^{l}}.
\]
Lucas' Theorem implies
\[
	\binom{N(n)+Dx+j}{N(n)} \equiv \binom{p^l(p^m-1)+Dx+j}{p^l(p^m-1)}\equiv 1\mod p.
\]
\end{proof}

Let us recall, from \Cref{main_notation}, that we have defined $Q:=p^{l+r+1}$ and $D:=d'p^l$ with $r:=\left\lfloor\nu_p\left(B_{2+\delta,\chi}\right)\right\rfloor+1$, and that $\delta\in\{0,1\}$ depends on the parity of $\chi$.

\begin{prop}\label{prop_intbinom}
	For a positive integer $j$, which is prime to $p$, define
	\[
		f_j(t):=\binom{N(n)+Dt+j}{N(n)}^{Q}\frac{(Dt+j)^{2+\delta}}{\prod_{i=0}^{n}(D(t+i)+j)^s}.
	\]
	If $(p-1)|s$ and  $p^{l+r}|(n+1)$, then
	\[
		\int_{\Zp}f_j(t) \mathrm{d}t\equiv j^{2+\delta}p^{-\lfloor \log_p(nd') \rfloor-1} \mod p^{-\lfloor \log_p(nd') \rfloor-1+l+r}.
	\]
\end{prop}
\begin{proof}Recall, $N(n)=p^l(p^m-1)$ with  $m:=\lfloor \log_p(nd') \rfloor+1$. We will make use of the basic properties of $\Delta(f)$ studied in \Cref{lem_df}. From \Cref{lem_df} $(e)$, we get
\[
	\Delta(\binom{p^l(p^m-1)+d't+j}{p^l(p^m-1)})\geq -l-m.
\]
Now, precomposing with multiplication by $p^l$ gives the inequality
\[
	\Delta(\binom{p^l(p^m-1)+d'p^lt+j}{p^l(p^m-1)})\geq -m,
\]
according to \Cref{lem_df} $(c)$. From \Cref{lem_df} $(d)$, \Cref{cor_chi} and $\mathcal{X}_{j,m}^Q=\mathcal{X}_{j,m}$, we get
\begin{align*}
	&\Delta\left(\binom{p^l(p^m-1)+d'p^lt+j}{p^l(p^m-1)}^Q-\mathcal{X}_{j,m}\right) \\ 
	\geq& \Delta\left(\binom{p^l(p^m-1)+d'p^lt+j}{p^l(p^m-1)}-\mathcal{X}_{j,m}\right)+l+r+1\geq  -m+l+r+1.
\end{align*}
The properties $(b)$ and $(f)$ of \Cref{lem_df} along with $m>l+r$ allow us to estimate
\[
	\Delta\left( f_j(t)- \mathcal{X}_{j,m}\cdot \frac{(Dt+j)^{2+\delta}}{\prod_{i=0}^{n}(D(t+i)+j)^s} \right)\geq -m+l+r+1.
\]
Applying \Cref{cor_Volkint} allows us to compute the Volkenborn integral of $f_j$ modulo $p^{-m+l+r}$:
\[
	\int_{\Zp}f_j(t)\mathrm{d}t\equiv \int_{\Zp} \mathcal{X}_{j,m}\cdot \frac{(Dt+j)^{2+\delta}}{\prod_{i=0}^{n}(D(t+i)+j)^s} \mathrm{d}t\equiv p^{-m}j^{2+\delta-(n+1)s}\equiv p^{-m}j^{2+\delta} \mod p^{-m+l+r}.
\]
Here, we have used the hypothesis $\gcd(j,p)=1$ and $p^{l+r}(p-1)|(n+1)s$.
 \end{proof}

Let us now recall the following well-known fact about generalized Bernoulli numbers:
\begin{lem}[{\cite[7.11]{washington}}]\label{lem_Bn}
	For a positive integer $k$, we have the congruence
	\[
		(1-\chi(p)p^{k-1})B_{k,\chi}\equiv \frac{1}{D}\sum_{\substack{1\leq j\leq D \\ \gcd(j,p)=1}}\chi(j)j^k \mod p^{l-1}.
	\]
\end{lem}
%In \cref{sec_1} we introduced the rational function
%\[
%	R_{n}(t):=n!^s \binom{N(n)}{\underline{n}}^Q\binom{Dt+N(n)}{N(n)}^Q\frac{(Dt)^{2+\delta}}{(t)_{n+1}^s},
%\]
%with $D=d'p^l,Q=p^{l+ r+1}$ and used it to construct linear forms in $p$-adic $L$-values. Our next aim is to estimate the $p$-adic valuation of these linear forms:
The following proposition is the main step in controlling the $p$-adic decay of the linear forms 
\[
	\Lambda_n\left(1,L_p(2,\chi\omega^{-1}),\dots,L_p(s+1,\chi\omega^{s})\right).
\]

\begin{prop}\label{prop_padicEstimate}
Let $s$ and $n$ be positive integers with $(p-1)\mid s$ and $p^{r+l}\mid (n+1)$. Then
\begin{multline*}
	\nu_p\left(\sum_{\substack{1\leq j\leq D \\ \gcd(j,p)=1}} \chi(j)\int_{\Zp} R_n(t+\frac{j}{D}) \mathrm{d}t \right)\\
	=s\nu_p(n!)+Q\nu_p\left(\binom{N(n)}{\underline{n}}\right)+((n+1)\cdot s+1)\cdot l -\lfloor \log_p (nd') \rfloor-1+\nu_p(B_{2+\delta,\chi}).
\end{multline*}
\end{prop}
\begin{proof}
We have
\begin{align*}
	\frac{1}{D\cdot n!^{s}\cdot D^{s\cdot (n+1)}\cdot \binom{N(n)}{\underline{n}}^Q}\sum_{\substack{1\leq j\leq D \\ \gcd(j,p)=1}} \chi(j)\int_{\Zp} R_n(t+\frac{j}{D}) \mathrm{d}t&=\frac{1}{D}\sum_{\substack{1\leq j\leq D \\ \gcd(j,p)=1}} \chi(j)\int_{\Zp} f_j(t) \mathrm{d}t
\end{align*}
with
\[
	f_j(t)=\binom{N(n)+Dt+j}{N(n)}^{Q}\frac{(Dt+j)^{2+\delta}}{\prod_{i=0}^{n}(D(t+i)+j)^s}.
\]
From \Cref{prop_intbinom} and \Cref{lem_Bn}, we know
\[
	\sum_{\substack{1\leq j\leq D \\ \gcd(j,p)=1}} \chi(j)\int_{\Zp} f_j(t) \mathrm{d}t\equiv p^{-m}(1-\chi(p)p^{1+\delta})B_{2+\delta,\chi} \mod p^{-m+l+r}
\]
with $m=\lfloor \log_p (nd')\rfloor+1$. We deduce
\[
	\frac{1}{D}\sum_{\substack{1\leq j\leq D \\ \gcd(j,p)=1}} \chi(j)\int_{\Zp} f_j(t) \mathrm{d}t\equiv p^{-m}(1-\chi(p)p^{1+\delta})B_{2+\delta,\chi} \mod p^{-m+r}.
\]
Taking $p$-adic valuations and observing $B_{2+\delta,\chi}\not\equiv 0 \mod p^r$ gives
\[
	\nu_p\left(\frac{1}{D}\sum_{\substack{1\leq j\leq D \\ \gcd(j,p)=1}} \chi(j)\int_{\Zp} f_j(t) \mathrm{d}t\right) =-m+\nu_p(B_{2+\delta,\chi})
\]
and the statement of the proposition follows.
\end{proof}

\Cref{prop_padicEstimate} only applies for integers $n$ such that $n+1$ is divisible by $p^{l+r}$. This forces us to pass to a suitable sub-sequence. Let us define 
\begin{equation}
	\sigma(n):=p^{l+r}n-1.
\end{equation}

\begin{cor}\label{cor_padic} Let $s$ be divisible by $(p-1)$. Then
\[
	\left| \Lambda_{\sigma(n)}\left(1,(L_p(i,\chi\omega^{1-i}))_{i=2}^{s+1}\right) \right|_p= \exp\left( -\sigma(n)\left(\tau_p(l,s)+o(1) \right) \right),
\]
as $n\rightarrow\infty$, with
\[
	\tau_p(l,s)=s\log p\left(l+\frac{1}{p-1}\right).%+QD,\quad \tau_2(l,s)=s\log p\left(l+\frac{1}{p-1}\right).
\]
\end{cor}
\begin{proof}
The $p$-adic valuation of
\[
d_{\sigma(n)}=\lcm(1,\dots,\sigma(n))=\prod_{p \text{ prime}} p^{\lfloor\log_p(\sigma(n))\rfloor}
\] 
is given by $\nu_p(\dn{\sigma(n)})=\lfloor \log_p(\sigma(n)) \rfloor$. By \eqref{cor_L_linforms} and \Cref{prop_padicEstimate}, we get the following estimate for the $p$-adic valuation
\begin{multline*}
	\nu_p\left(\Lambda_{\sigma(n)}\left(1,L_p(2,\chi\omega^{-1}),\dots,L_p(s+1,\chi\omega^{s})\right)\right)\\
	= s\nu_p(\sigma(n)!)+Q\nu_p\left(\binom{N(\sigma(n))}{\underline{\sigma(n)}}\right)+\sigma(n) s l + o(n),\quad \text{as }n\rightarrow\infty.
\end{multline*}
For a positive integer $m$, we have, by Legendre's formula:
\[
	\frac{m}{p-1}-\log_p(m)\leq\nu_p(m!)\leq \frac{m}{p-1}.
\]
In particular, we get
\begin{multline*}
	0\leq \nu_p\left(\binom{N(n)}{\underline{n}}\right)=\nu_p\left(\frac{N(n)!}{n!^{\lfloor N(n)/n \rfloor}(N(n)-n\lfloor N(n)/n \rfloor)!}\right)\\
	\leq \frac{N(n)}{p-1}-\left\lfloor \frac{N(n)}{n} \right\rfloor\left( \frac{n}{p-1}-\log_p\left(n \right) \right)-\left( \frac{N(n)-n \left\lfloor \frac{N(n)}{n} \right\rfloor }{p-1} -\log_p\left( N(n)-n \left\lfloor \frac{N(n)}{n} \right\rfloor \right)\right)\\
	=\left\lfloor \frac{N(n)}{n} \right\rfloor \log_p\left(n \right)+\log_p\left( N(n)-n \left\lfloor \frac{N(n)}{n} \right\rfloor \right)\leq (Dp+1)\log_p(n)=o(n)  ,\quad \text{as }n\rightarrow \infty.
\end{multline*}
Hence the only contribution comes from the terms $s\nu_p(\sigma(n)!)$ and $\sigma(n) s l$, as $n\rightarrow \infty$:
\begin{equation*}
\nu_p\left(\Lambda_{\sigma(n)}\left(1,L_p(2,\chi\omega^{-1}),\dots,L_p(s+1,\chi\omega^{s})\right)\right)=\sigma(n)\left[ sl+ \frac{s}{p-1} +o(1) \right].
\end{equation*} 
%
%Together with \Cref{prop_padicEstimate} we get:
%\[
%	\exp\left( -\sigma(n)\left( \tau_1(l,s)+o(1) \right) \right)\leq\left| \Lambda_{\sigma(n)}\left(1,(L_p(i,\chi\omega^{1-i}))_{i=2}^{s+1}\right) \right|_p\leq \exp\left( -\sigma(n)\left( \tau_2(l,s)-o(1) \right) \right)
%\]
%as $n\rightarrow \infty $ with
%\[
%	\tau_1(l,s)=s\log p\left(l+\frac{1}{p-1}\right)+QD,\quad \tau_2(l,s)=s\log p\left(l+\frac{1}{p-1}\right).
%\]
\end{proof}

\section{Proof of the main result}\label{sec_2}
Let $\chi$ be a Dirichlet character of conductor $d$, $\epsilon>0$ a real number and $K\subseteq \bar{\QQ}$ a number field. We may assume, without loss of generality, that $\QQ(\chi)$ is contained in $K$. In this section, we prove \Cref{thm_main}, i.e. we show the following estimate for $s$ sufficiently large
\[
	\dim_{K}\left( K+\sum_{i=2}^{s}L_p(i,\chi\omega^{1-i})K  \right)\geq \frac{(1-\epsilon)\log s}{2[K:\QQ]\left( 1+\log(2)\right)}.
\]
%For the convenience of the reader let us summarize the notation which has been used throughout this paper. We have $d=d'p^{l_0}$ with $d'$ coprime to $p$. Depending on the parity of $\chi$ we have defined 
%\[
%	\delta:=\begin{cases} 0, & \text{if } \chi(-1)=1\\
%	1, & \text{if }  \chi(-1)=-1, \end{cases}
%\]
%and
%\[
%	r:=\lfloor\nu_p\left(B_{2+\delta,\chi}\right)\rfloor+1.
%\]
%Until now, $l\in \ZZ_{>1}$ was an arbitrary integer with the only constraint $l\geq l_0$. 
We keep the setup from \Cref{main_notation}. Our next aim is to choose the parameter $l$ suitably depending on $\chi$ and $s$. Let us denote by $W$ the inverse of the function
$$\mathbb{R}_{>0}\rightarrow \mathbb{R}_{>0},\quad x\mapsto x\exp(x).$$
This function is called the \emph{Lambert $W$-function}.
\begin{lem}\label{lem_Lambert}
Let us define the function $\ell \colon \ZZ_{\geq 0}\rightarrow \ZZ_{\geq 0}$ by
\begin{equation}\label{eq_l}
	\ell(s):=\left\lfloor \frac{W(\frac{2s\cdot \epsilon}{3d'p^{r+2}})}{2\log p} \right\rfloor.
\end{equation}
For all sufficiently large integers $s$, we have the estimate
\begin{equation}\label{eq_ls}
	\frac{\tau_p(\ell(s),s)}{\tau_\infty(\ell(s),s)}\geq (1-\epsilon)\frac{\log s}{2[K:\QQ](1+\log(2))},
\end{equation}
where $\tau_p(\ell(s),s)$ and $\tau_\infty(\ell(s),s)$ are the quantities appearing in \Cref{cor_padic} and \Cref{cor_growth} for $l:=\ell(s)$, i.e.
\begin{align*}
%	\tau_1(\ell(s),s)&=s\log p\left(\ell(s)+\frac{1}{p-1}\right)+d'p^{r+1+2\ell(s)}\\
	\tau_p(\ell(s),s)&=s\log p\left(\ell(s)+\frac{1}{p-1}\right)\\
	\tau_\infty(\ell(s),s)&=[K:\QQ]\left( s\log 2+d'p^{r+2+2\ell(s)}\log(d'p^{1+\ell(s)})+s\right).
\end{align*}
\end{lem}
\begin{proof}
Let us start with estimating the term
\begin{equation}\label{eq_Lem_Lambert1_1}
	d'p^{r+2+2\ell(s)}\log(d'p^{1+\ell(s)})=d'\ell(s)p^{r+2+2\ell(s)}\log (p)+d'p^{r+2+2\ell(s)}\log (d'p),
\end{equation}
appearing in $\tau_\infty(\ell(s),s)$. The first summand can be bounded as follows:
\begin{align}
\notag d'\ell(s)p^{r+2+2\ell(s)}\log (p)&=\frac{d'p^{r+2}}{2}2\log (p)\ell(s) \exp\left(2\log (p)\ell(s) \right)\\
\label{eq_Lem_Lambert1} &\leq \frac{d'p^{r+2}}{2}2\log (p) \left( \frac{W(\frac{2s\cdot \epsilon}{3d'p^{r+2}})}{2\log (p)} \right)  \exp\left(2 \log (p)\left( \frac{W(\frac{2s\cdot \epsilon}{3d'p^{r+2}})}{2\log (p)} \right) \right)\\
\notag &=\frac{d'p^{r+2}}{2}W\left(\frac{2s\cdot \epsilon}{3d'p^{r+2}}\right) \exp\left(W\left(\frac{2s\cdot \epsilon}{3d'p^{r+2}}\right) \right)=\frac{\epsilon}{3}s.
\end{align}
This estimate together with \eqref{eq_Lem_Lambert1_1} gives
\[
	d'p^{r+2+2\ell(s)}\log(d'p^{1+\ell(s)})\leq \frac{\epsilon}{3}s+o(s) \text{ as } s\rightarrow\infty,
\]
and allows us to bound $\tau_\infty(\ell(s),s)$:
\begin{align*}
	%&\tau_1(\ell(s),s)-\tau_2(\ell(s),s)+\tau_\infty(\ell(s),s)\\
	%=&[K:\QQ]\left(s+ s\log 2+d'p^{r+2+2\ell(s)}\log(d'p^{1+\ell(s)}) +\frac{d'p^{r+1+2\ell(s)}}{[K:\QQ]}\right)\\
	\tau_\infty(\ell(s),s)\leq [K:\QQ]\left(s+ s\log 2+\frac{\epsilon}{3}s+o(s)\right) \text{ as } s\rightarrow\infty.
\end{align*}
So, for $s$ sufficiently large, we get
\begin{equation}\label{eq_lem_Lambert1}
\frac{s}{\tau_\infty(\ell(s),s)}\geq \frac{1-\frac{2}{3}\epsilon}{[K:\QQ](1+\log2)}.
\end{equation}
On the other hand, we have the following well-known estimate for the Lambert $W$-function
\[
	W(x)\geq \log(x)-\log\log(x),\text{ for $x$ sufficiently large}.
\]
For $s$ sufficiently large, this gives the inequality
\begin{equation}\label{eq_lem_Lambert2}
	\ell(s)\geq \left(1-\frac{\epsilon}{3}\right)\frac{\log s}{2\log p}.
\end{equation}
Now the claim follows by combining \eqref{eq_lem_Lambert1} and \eqref{eq_lem_Lambert2}
\begin{align*}
		\frac{\tau_p(\ell(s),s)}{\tau_\infty(\ell(s),s)}&\geq \frac{s\log (p) \ell(s)}{\tau_\infty(\ell(s),s)} \geq \frac{\log s\left(1-\frac{1}{3}\epsilon\right) \left(1-\frac{2}{3}\epsilon\right)}{2[K:\QQ](1+\log2)} \geq \frac{(1-\epsilon)\log s }{2[K:\QQ](1+\log2)}.
\end{align*}
\end{proof}

Let us assume that $s$ is sufficiently large such that the inequality \eqref{eq_ls} holds. Furthermore, we may assume, without loss of generality, that $p-1$ divides $s$. Define $l:=\ell(s)$, where $\ell$ is the function appearing in \Cref{lem_Lambert}. From now on, $Q$, $D$ and $R_n$ refer to the quantities defined in \eqref{def_QD} and \eqref{eq_def_Rn} with the specific choice $l=\ell(s)$.

In \eqref{cor_L_linforms}, we have shown the formula
\[
(s-1)!\dn{n}^{s-1}\sum_{\substack{1\leq j\leq D \\ \gcd(j,p)=1}} \chi(j)\int_{\Zp} R_n(t+\frac{j}{D}) \mathrm{d}t=\Lambda_n\left(1,L_p(2,\chi\omega^{-1}),\dots,L_p(s+1,\chi\omega^{s})\right).
\]
\Cref{cor_integrality} shows the integrality of these linear forms, i.e.
\[
		\Lambda_n(X_0,\dots,X_s)\in \Ocal_{K} X_0+\dots+ \Ocal_{K} X_s.
\]
The Archimedean norm of these linear forms is controlled by \Cref{cor_growth}:
\[
		H_K(\Lambda_n)\leq \exp\left( n\left[\tau_\infty(l,s)+o(1)\right] \right) \text{ as } s\rightarrow \infty.
\]
For the subsequence $(\sigma(n))_{n\geq 1}$ with $\sigma(n):=p^{l+r}n-1$, we have proven in \Cref{cor_padic} the estimate
\[
	\left| \Lambda_{\sigma(n)}\left(1,(L_p(i,\chi\omega^{1-i}))_{i=2}^{s+1}\right) \right|_p = \exp\left( -\sigma(n)\left(\tau_p(l(s),s)+o(1) \right) \right).
\]
Applying Nesterenko's $p$-adic linear independence criterion (\Cref{thm_nesterenko}) to the sequence of linear forms $(\Lambda_{\sigma(n)})_n$, gives the estimate
\[
	\dim_{K}\left( K+\sum_{i=2}^{s+1}L_p(i,\chi\omega^{1-i})K  \right)\geq \frac{\tau_p(\ell(s),s)}{\tau_\infty(\ell(s),s)}\geq (1-\epsilon)\frac{\log s}{2[K:\QQ](1+\log(2))}.
\]
%Finally the inequality  gives the desired estimate
%\[
%	\dim_{K}\left( K+\sum_{i=2}^{s+1}L_p(i,\chi\omega^{1-i})K  \right)
%\]

\section{\texorpdfstring{The proof for $p$-adic Hurwitz zeta values}{The proof for p-adic Hurwitz zeta values}}
The aim of this section is to indicate the proof of \Cref{thm_Hurwitz}, i.e. for given $x\in \QQ$ with $|x|_p>1$, $\epsilon>0$ and a sufficiently large integer $s$, we want to prove the inequality
\[
	\dim_{K}\left( K+ \zeta_p(2,x)K + \dots + \zeta_p(s,x)K  \right)\geq \frac{(1-\epsilon)\log s}{2[K:\QQ](1+\log(2))}.
\]
The proof of this theorem follows the same lines as the proof of \Cref{thm_main}. Let us briefly outline the argument. The Teichm\"uller character $\omega$ takes values in the field $\QQ(\mu_{(p-1)})$ obtained by adjoining all $(p-1)$st roots of unity to $\QQ$. We may assume $\QQ(\mu_{(p-1)})\subseteq K$. Furthermore, the formula \cite[Thm. 11.2.9. (3)]{cohen2}
\[
	\zeta_p(i,x+1)-\zeta_p(i,x)=-\frac{\langle x \rangle^{1-i}}{x}=-\frac{\omega(x)^{i-1}}{x^i}\in K.
\]
allows us to assume $0<x \leq 1$, without loss of generality.
In particular, let us write $x=\frac{j_0}{d}$ with $\gcd(j_0,p)=1$ and $1\leq j_0\leq d$. Let us write
\[
d=d'p^{l_0}\text{ with } \gcd(d',p)=1.
\]
 By the assumption $|x|_p>1$, we have $l_0\geq 1$. Let us set
\[
	l:=\ell(s)=\left\lfloor \frac{W(\frac{2s\cdot \epsilon}{3d'p^{2}})}{2\log p} \right\rfloor,
\]
where $W$ denotes the Lambert $W$-function, see Section \ref{sec_2}. For $s$ sufficiently large, we may assume $l\geq \max(2,l_0)$. Let us define $D:=p^ld'$, $Q:=p^{l+1}$ and
\[
	R_n(t):=n!^{s}\binom{N(n)}{\underline{n}}^Q\binom{N(n)+Dt}{N(n)}^Q\frac{1}{(t)^s_{n+1}}
\]
with $N(n):=p^l(p^{\lfloor \log_p(d'n) \rfloor+1}-1)$. Note that this function is a particular case of the rational function $R_n(t)$, defined in \Cref{main_notation}, if we set $\delta=-2$ and $r=0$. Applying \Cref{prop_linearforms} gives for each $\tilde{x}\in\Qp$ with $|\tilde{x}|_p\geq p$ (resp. $|\tilde{x}|_p\geq 4$ if $p=2$) the formula
	\begin{equation}\label{eq_Hurwitz_integration}
		\int_{\Zp} R_n(t+\tilde{x}) \mathrm{d}t=\rho^{(n)}_{0,\tilde{x}}+\sum_{i=1}^s \rho_i^{(n)}\omega(\tilde{x})^{-i}\zeta_p\left(i+1,\tilde{x}\right).
	\end{equation}
Together with \eqref{eq_Hurwitz_integration}, we get linear forms in $p$-adic Hurwitz zeta values:
\[
	\sum_{0\leq j<p^{l-l_0}}\int_{\Zp}R_n(t+\frac{j_0}{D}+\frac{j}{p^{l-l_0}})\mathrm{d}t=\left(\sum_{0\leq j<p^{l-l_0}}\rho_{0,(j_0+dj)/D}^{(n)}\right)+\sum_{i=1}^{s}\rho_{i}^{(n)}\omega(j_0)^{-i}D^i \zeta_p(i+1,x).
\]
Let us use this formula to define our linear forms $\tilde{\Lambda}_n(X_0,\dots,X_s)$ in $p$-adic Hurwitz zeta values:
\begin{align*}
	\tilde{\lambda}_0:&=(s-1)!\dn{n}^{s-1}\sum_{0\leq j<p^{l-l_0}}\rho_{0,(j_0+dj)/D}^{(n)}\\
	\tilde{\lambda}_i:&=(s-1)!\dn{n}^{s-1} \rho_i^{(n)}\omega(j_0)^{-i}D^{i}\quad 1\leq i\leq s.
\end{align*}
Applying \Cref{prop_arith}, with $\delta=-2$ and $r=0$, shows the integrality of the linear forms $(\tilde{\Lambda}_n)_n$, i.e. we have
\[
	\tilde{\Lambda}_n(X_0,\dots,X_s)\in X_0\Ocal_K+\dots +X_s \Ocal_K.
\]
\Cref{prop_growth}, with $\delta=-2$ and $r=0$, controls the asymptotic growth of the coefficients of $\tilde{\Lambda}_n$ and we obtain
\[
	H_K(\tilde{\Lambda}_n)\leq \exp\left( n\left[ \tau_\infty(\ell(s),s)+o(1) \right] \right)\text{ as }n\rightarrow\infty
\]
with
\[
	\tau_\infty(\ell(s),s)=[K:\QQ]\left( s+s\log 2 +d'p^{2\ell(s)+2}\log p \right).
\]
Let us introduce the function $\sigma(n)\colon \ZZ_{\geq 1}\rightarrow\ZZ_{\geq 1}, n\mapsto p^ln-1$.
\begin{prop}
Let $s$ be divisible by $p-1$. Then
\[
	|\tilde{\Lambda}_{\sigma(n)}(1,\zeta_p(2,x),\dots,\zeta_p(s+1,x))|_p= \exp\left( -\sigma(n)(  \tau_p(\ell(s),s)+o(1))  \right),
\]
as $n\rightarrow\infty$, with
\[
	\tau_p(\ell(s),s)=s\log p\left(\ell(s)+\frac{1}{p-1}\right).
\]
\end{prop}
\begin{proof}
By definition of $R_n(t)$, we have for each integer $n$ the equality
\begin{align*}
	\tilde{\Lambda}_n(1,\zeta_p(2,x),\dots,\zeta_p(s+1,x))&=(s-1)!\dn{n}^{s-1}n!^{s}D^{s(n+1)}\binom{N(n)}{\underline{n}}^Q\cdot \sum_{0\leq j<p^{l-l_0}}\int_{\Zp}f_{j_0+jd}(t) \mathrm{d}t
\end{align*}
with
\[
	f_{j_0+jd}(t)=\binom{N(n)+Dt+j_0+jd}{N(n)}^{Q}\frac{1}{\prod_{i=0}^{n}(D(t+i)+j_0+jd)^s}.
\]
If $n$ is an integer satisfying $p^r|n+1$, we deduce from \Cref{prop_intbinom} the congruence
\[
	\int_{\Zp}f_{j_0+jd}(t) \mathrm{d}t\equiv p^{-m}\mod p^{-m+l}
\]
for $m:=\lfloor \log_p(nd') \rfloor+1$. This gives
\[
	\sum_{0\leq j<p^{l-l_0}}\int_{\Zp}f_{j_0+jd}(t) \mathrm{d}t\equiv p^{-m+l-l_0}\mod p^{-m+l}
\]
and the desired estimate follows along the same lines as in the proof of \Cref{cor_padic}.
\end{proof}
Applying the $p$-adic Nesterenko criterion and observing \Cref{lem_Lambert}, gives:
\[
	\dim_K\left( K+ \zeta_p(2,x)K+\dots+\zeta_p(s+1,x)K\right)\geq \frac{\tau_p(\ell(s),s)}{\tau_\infty(\ell(s),s)}\geq \frac{(1-\epsilon)\log s}{[K:\QQ](2+2\log(2))}.
\]

\section{Proof of the linear independence criterion}\label{ch_linearindependence}

The aim of this section is to provide a proof of the $p$-adic linear independence criterion, formulated in \Cref{thm_nesterenko}. The main difference to Nesterenko's criterion, in \cite{nesterenko_p}, is the missing renormalization of the $p$-adic norm of the linear form. The present criterion is proven in the case $K=\QQ$ by Chantanasiri, in \cite{chantanasiri}, see also \cite[II, \S 6]{colmez}. There, it is indicated that the proof works for arbitrary number fields. For the convenience of the reader, we include a proof for an arbitrary number field $K$.\par
Let us recall that the height of a linear form
	\[
		L(\underline{X})=l_{0}X_0+l_{1}X_1+\dots+l_{r}X_r,\quad l_{i}\in K
	\]
	with coefficients in a number field $K$ was defined as follows
	\[
		H_K(L):=\max_{0\leq i\leq r} |N_K(l_{i})|.
	\]
	In the following, we identify $L$ with the $1\times (r+1)$-matrix $(l_0,...,l_r)$, and we extend the definition of the height to more general matrices. Let $M=(m_{i,j})_{\substack{ 0\leq i\leq s \\ 0\leq j\leq r }}\in M_{(s+1)\times(r+1)}(K)$ be a $(s+1)\times(r+1)$-matrix with  $0\leq s\leq r$. We define the \emph{height of $M$} by
\[
	H_K(M):=\max_{J}|N_K(\det M_J)|,
\]
where $J$ runs through all subsets $J\subseteq \{0,\dots,r\}$ of cardinality $s+1$ and $M_J$ denotes the matrix $M_J=(m_{i,j})_{0\leq i\leq s, j\in J}$.
The following lemma is deduced immediately from the definitions of $H_K$ and $\Delta$, cf. \cite[Lemme 1.4]{chantanasiri}:
\begin{lem}\label{lem_chan_1}
The above defined height satisfies the following properties:
\begin{enumerate}
\item \label{lem_chan_1_i}  If $s=0$ and $M=(m_0,\dots,m_r)$, then
\[
	H_K(M)=\max_{0\leq i\leq r} |N_K(m_i)|.
\]
\item \label{lem_chan_1_ii}  If $s=r$, i.e. $M\in M_{(r+1)\times(r+1)}(K)$, then
\[
	H_K(M)=|N_K(\det(M))|.
\]
\item \label{lem_chan_1_iii}  If $s\leq r-1$ and $M\in M_{(s+1)\times(r+1)}(K)$, $L\in M_{1\times (r+1)}(K)$, let us denote by $M\oplus L$ the matrix obtained by appending $L$ to $M$ as the last row, then
\[
	H_K(M\oplus L)\leq (s+2)H_K(M)H_K(L).
\]
\end{enumerate}
\end{lem}
Let us now assume that we have fixed an embedding $K\subseteq \mathbb{C}_p$ and let us define
\[
	H_p(M):=\max_{J}|\det M_J|_p,
\]
where $J$ runs through all subsets $J\subseteq \{0,\dots,r\}$ of cardinality $s+1$.
For $\underline\xi=(\xi_0,\dots\xi_r)\in \Cp^{r+1}$ and $J'\subseteq\{ 0,\dots,r\}$ with $|J'|=s$, let us write $M_{J',\underline\xi}\in M_{(s+1)\times(s+1)}(K)$ for the matrix obtained by appending the column vector $M\underline\xi$ to the matrix $M_{J'}$. Here, we view $\underline{\xi}$ as a column vector and write $M\underline\xi$ for the image of $\underline{\xi}$ under $M$. Let us define
\[
	\Delta_p(M):=\max_{J'}|\det M_{J',\underline\xi}|_p,
\]
where $J'\subset \{0,\dots,r\}$ runs through the subsets of $\{0,\dots,r\}$ with $s$ elements. We observe the following elementary properties of $H_p$ and $\Delta_p$:
\begin{lem}[{cf. \cite[Lemme 1.5]{chantanasiri}}]\label{lem_chan_2}
Let $M\in M_{(s+1)\times(r+1)}(\Ocal_K)$ with $0\leq s\leq r$.
\begin{enumerate}
\item \label{lem_chan_2_i} If $s=r$, then
\[
	H_p(M)=|\det(M)|_p \text{ and } \Delta_p(M)=|\det M|_p\cdot \max_{0\leq i\leq r}|\xi_i|_p.
\]
\item \label{lem_chan_2_ii}  If $\Delta_p(M)\neq 0$, then $H_K(M)\neq 0$ and $H_p(M)\geq \frac{1}{H_K(M)}$.
\item \label{lem_chan_2_iii}  If $s\leq r-1$ and $L\in M_{1\times (r+1)}(K)$, then
\[
	H_p(M\oplus L)\leq (s+2)H_p(M)H_p(L).
\]
If furthermore $H_p(M)\Delta_p(L)>H_p(L)\Delta_p(M)$, then
\[
	\Delta_p(M\oplus L)=H_p(M)\Delta_p(L).
\]
\end{enumerate}
\end{lem}

\begin{thm}[{\cite[Thm. 1.7]{chantanasiri}}]\label{thm_chantanasiri}
 Let $\underline\xi=(\xi_0,\dots,\xi_r)\in\Cp^{r+1}$ and let $(A_n)_{n\geq n_0}$, $(B_n)_{n\geq n_0}$, $(Q_n)_{n\geq n_0}$ be sequences of real positive numbers. Assume that the sequence $(B_nQ_{n-1})_{n>n_0}$ is increasing, and that $(A_n)_{n\geq n_0}$ and $(B_nQ_{n-1})_{n\geq n_0}$ tend to infinity. For a sequence of linear forms
 \[
 	L_n(\underline X)=l_{0,n} X_0+\dots+l_{r,n}X_r,\quad l_{j,n}\in\Ocal_K
 \]
 with
 \[
 	H_K(L_n)\leq Q_n,\quad 0<|L_n(\underline{\xi})|_p\leq \frac{1}{A_n} \quad\text{and}\quad \frac{|L_{n-1}(\underline{\xi})|_p}{|L_n(\underline{\xi})|_p}\leq B_n,
 \]
 there exists a positive constant $c>0$ such that for all sufficiently large integers $n$: $$c \cdot A_n<(r+1)!(B_nQ_{n-1})^{r}Q_n.$$
\end{thm}
\begin{proof}
Let us sketch the proof following \cite[Thm. 1.7]{chantanasiri}. By abuse of notation, let us denote the $1\times(r+1)$-matrix $(l_{0,n},\dots,l_{r,n})$ associated to the linear form $L_n$ again by $L_n$. For a sufficiently large positive number $n$ and $i \leq r$, let us inductively construct a sequence
\[
	n=k_0>k_1>\dots>k_i\geq n_0,
\]
such that the matrix
\[
	M_n^{(i)}:=L_{k_0}\oplus \dots \oplus L_{k_i}
\]
satisfies
\begin{equation}\label{thm_chan_eq1}
	0<\Delta_p(M_n^{(i)})\leq B_{k_1+1}\cdot \dots \cdot B_{k_i+1} \frac{1}{A_n}.
\end{equation}
For $i=0$, let us set $k_0=n$. We have by the hypothesis of \Cref{thm_chantanasiri} the inequality $\Delta_p(M_n^{(0)})=|L_n(\underline\xi)|_p\leq \frac{1}{A_n}$. For $0\leq i\leq r-1$, suppose that $k_0,\dots,k_i$ have already been chosen. According to \Cref{lem_chan_1} \ref{lem_chan_1_iii} and the hypothesis, we have
\begin{equation}\label{thm_chan_eq3}
	H_K(M_n^{(i)})\leq (i+1)!Q_{k_0}\cdot\dots\cdot Q_{k_i}.
\end{equation}
Let us first assume that
\begin{equation}\label{thm_chan_eq2}
	\Delta_p(M_n^{(i)})H_p(L_{n_0})\geq |L_{n_0}(\underline\xi)|_p H_p(M_n^{(i)}).
\end{equation}
In this case, we can skip the inductive construction and argue directly as follows: From \eqref{thm_chan_eq1} and \eqref{thm_chan_eq2} we get
\[
	A_n\leq B_{k_1+1}\cdot \dots\cdot B_{k_i+1}\frac{H_p(L_{n_0})}{|L_{n_0}(\underline{\xi})|_p H_p(M_n^{(i)})}.
\]
Now, we can make use of \eqref{thm_chan_eq3} and \Cref{lem_chan_2}  \ref{lem_chan_2_ii}
\[
	\frac{1}{H_p(M_n^{(i)})}\leq H_K(M_n^{(i)})\leq (i+1)!Q_{k_0}\cdot \dots \cdot Q_{k_i}
\]
and deduce
\[
	A_n\leq (i+1)!B_{k_1+1}\cdot \dots\cdot B_{k_i+1}Q_{k_0}\cdot \dots \cdot Q_{k_i} \frac{H_p(L_{n_0})}{|L_{n_0}(\underline{\xi})|_p }.
\]
By the hypothesis of the theorem, the sequence $B_nQ_{n-1}$ is increasing and in this case we conclude
\[
	\frac{|L_{n_0}(\underline{\xi})|_p }{H_p(L_{n_0})}A_n\leq (i+1)!Q_n(B_nQ_{n-1})^r,
\]
as desired. Let us resume our inductive construction of $k_{i+1}$ under the assumption that \eqref{thm_chan_eq2} does not hold. In this case, the set
\[
	S=\left\{ k\in \ZZ \mid n_0\leq k\leq k_i \text{ and } \Delta_p(M_n^{(i)})H_p(L_{k})< |L_{k}(\underline\xi)|_p H_p(M_n^{(i)}) \right\}
\]
is not empty, since it contains $n_0$. Let us set $k_{i+1}:=\max( S)$ and $M_n^{(i+1)}=M_n^{(i)}\oplus L_{k_{i+1}}$. \Cref{lem_chan_2}\ref{lem_chan_2_iii} gives
\[
	\Delta_p(M_n^{(i+1)})=H_p(M_n^{(i)})|L_{k_{i+1}}(\underline{\xi})|_p.
\]
In particular, $\Delta_p(M_n^{(i+1)})\neq 0$ which implies $k_{i+1}<k_i$. By the choice of $k_{i+1}$, we get the inequality
\[
	|L_{k_{i+1}+1}(\underline{\xi})|_p H_p(M_n^{(i)})\leq \Delta_p(M_n^{(i)})H_p(L_{k_{i+1}+1}).
\]
Combining the last two inequalities and observing $H_p(L_{k_{i+1}+1})\leq 1$, gives
\[
	\Delta_p(M_n^{(i+1)})\leq \Delta_p(M_n^{(i)})H_p(L_{k_{i+1}+1})\frac{|L_{k_{i+1}}(\underline{\xi})|_p}{|L_{k_{i+1}+1}(\underline{\xi})|_p}\leq \Delta_p(M_n^{(i)})B_{k_{i+1}+1}.
\]
This implies that condition \eqref{thm_chan_eq1} also holds for $i+1$ and finishes the inductive construction.\par
Thus we may assume that we have constructed
\[
	n=k_0>k_1>\dots>k_r\geq n_0
\]
such that
\begin{equation*}
	0<\Delta_p(M_n^{(r)})\leq B_{k_1+1}\cdot \dots \cdot B_{k_r+1} \frac{1}{A_n}.
\end{equation*}
According to \Cref{lem_chan_2}\ref{lem_chan_2_i}, we have $\Delta_p(M_n^{(r)})=|\det M_n^{(r)}|_p(\max_{i}|\xi_i|_p)$. Since $\det M_n^{(r)}\in K^\times$, we get by the product-formula for normed places of number fields
\[
	 |\det M_n^{(r)}|_p\geq \frac{1}{|N_K(\det M_n^{(r)})|}.
\]
Now, we make use of \Cref{lem_chan_1}\ref{lem_chan_1_ii} and \eqref{thm_chan_eq3} and get
\[
	|N_K(\det M_n^{(r)})|=H_K(M_n^{(r)})\leq (r+1)!Q_{k_0}\cdot \dots\cdot Q_{k_r},
\]
and then
\[
	\max_{j}|\xi_j|_p A_n\leq (r+1)!Q_n(B_{k_1+1}Q_{k_1})\cdot \dots\cdot (B_{k_r+1}Q_{k_r} ).
\]
Since $(B_{k+1}Q_k)_{k\geq n_0}$ is increasing, we get
\[
	\max_{j}|\xi_j|_p A_n\leq (r+1)!Q_n(B_nQ_{n-1})^r
\]
as desired.
\end{proof}

Let us finally deduce the linear independence criterion, stated in the introduction, from \Cref{thm_chantanasiri}:

\begin{proof}[Proof of \Cref{thm_nesterenko}]Let $\underline{\theta}=(\theta_1,\dots,\theta_s)\in\Cp^s$ and $(\Lambda_n)_n$ a sequence of linear forms satisfying
	\[
		H_K(\Lambda_n)\leq e^{\sigma(n)(\tau+o(1))} \quad\text{ and } e^{-(\tau_1+o(1))\sigma(n)}\leq |\Lambda_n(1,\underline{\theta})|_p\leq e^{-(\tau_2-o(1))\sigma(n)},
	\]
	as in the statement of \Cref{thm_nesterenko}. Our aim is to show the estimate
	\[
	\dim_K(K+K\theta_1+\dots+K\theta_s)\geq \frac{\tau_2}{\tau+\tau_1-\tau_2}.
	\]
	Let $\{\xi_0,\dots,\xi_r\}$ with $\xi_0=1$ be a $K$-basis of the $K$-vector space spanned by $\theta_0:=1,\theta_1,\dots,\theta_s$. There is an integer $M\in \ZZ$ and algebraic integers $c_{i,0},\dots,c_{i,r}\in\Ocal_K$ such that
	\[
		M\theta_i=\sum_{j=0}^r c_{i,j}\xi_j \text{ for all } 1\leq i\leq s.
	\]
	Let us define the linear form
	\[
		L_n(Y_0,\dots,Y_r):=\Lambda_n(\sum_{j=0}^r c_{0,j}Y_j,\dots,\sum_{j=0}^r c_{s,j}Y_j).
	\]
	Let us apply \Cref{thm_chantanasiri} with $\tau'>\tau$, $\tau_1'>\tau_1$ and $\tau_2'<\tau_2$ to $\underline{\xi}$, $L_n$ with the choice
	\[
		Q_n=Ce^{\sigma(n)\tau'}, A_n=e^{\tau_2'\sigma(n)}\quad \text{ and }B_n=e^{\tau'_1\sigma(n)-\tau_2'\sigma(n)}
	\]
	where $C>0$ is a suitable constant: There is a constant $c>0$ such that for all sufficiently large $n$:
	\[
		c e^{\tau_2'\sigma(n)}\leq  (r+1)!C^{r+1}\cdot (e^{\tau_1'\sigma(n)-\tau_2'\sigma(n)+\tau'\sigma(n-1)})^r e^{\tau'\sigma(n)}.
	\]
	By the hypothesis, we have $\lim_{n\rightarrow\infty}\sigma(n)=\infty$ and $\lim_{n\rightarrow\infty}\frac{\sigma(n)}{\sigma(n+1)}=1$ and deduce 
	\[
	\tau_1'\leq (r+1)(\tau'_1-\tau'_2+\tau').
	\]
	Since $\tau'>\tau$, $\tau_1'>\tau_1$ and $\tau_2'<\tau_2$ are arbitrary, we obtain the desired estimate:
	\[
	\dim_K(K+\sum_{i=1}^s \theta_i K)=r+1\geq \frac{\tau_1}{\tau+\tau_1-\tau_2}.
	\]
\end{proof}

\bibliographystyle{amsalpha}
\bibliography{padicZeta}
\end{document}